\documentclass[11pt,a4paper]{article}
\usepackage{amsfonts,amssymb,amsmath,amsthm}
\usepackage{epsfig} 
\usepackage{psfrag}
\usepackage{graphicx}
\usepackage{enumerate}
\usepackage{xcolor}
\usepackage{tikz}
\usepackage{pgfplots}
\usepackage{ulem}

\newtheorem{theorem}{Theorem}[section]

\newtheorem{corollary}[theorem]{Corollary}
\newtheorem{definition}[theorem]{Definition}
\newtheorem{example}[theorem]{Example}
\newtheorem{lemma}[theorem]{Lemma}
\newtheorem{proposition}[theorem]{Proposition}
\newtheorem{remark}[theorem]{Remark}

\def\CC{\mathbb C}

\def\ZZ{\mathbb Z}

\def\NN{\mathbb N}
\def\RR{\mathbb R}

\def\LL{\mathcal L}

\def\EE{\mathcal E}

\def\I{\mathcal{I}}

\newcommand{\carl}{\mathcal{C}}

\newcommand{\re}{\mathop{\rm Re}\nolimits}
\renewcommand{\Re}{\re}
\newcommand{\im}{\mathop{\rm Im}\nolimits}
\renewcommand{\Im}{\im}

\def\text{\mbox}

\title{Laplace--Carleson embeddings and infinity-norm admissibility}
\author{Birgit Jacob\thanks{Fakult\"at f\"ur  Mathematik und Naturwissenschaften, IMACM,
Bergische Universit\"at Wuppertal, Gau\ss stra\ss e 20, 42119 Wuppertal, Germany, \tt bjacob@uni-wuppertal.de}
\qquad Jonathan R.~Partington\thanks{School of Mathematics,
University of Leeds,
Leeds LS2 9JT, U.K. \tt
J.R.Partington@leeds.ac.uk}
\qquad Sandra Pott\thanks{Centre for Mathematics, Faculty of Science, Lund University, S\"olvegatan 18, 22100 Lund, Sweden, \tt Sandra.Pott@math.lu.se}
 \\ \\\qquad Eskil Rydhe\thanks{Centre for Mathematics, Faculty of Science, Lund University, S\"olvegatan 18, 22100 Lund, Sweden, {\tt Eskil.Rydhe@math.lu.se}.} \qquad Felix L.~Schwenninger\thanks{Department of Applied Mathematics, University of Twente, P.O. Box 217, 7500 AE Enschede, The Netherlands and Department of Mathematics, University of Hamburg, Bundesstra\ss e 55, 20146 Hamburg, Germany, 
              \tt{f.l.schwenninger@utwente.nl}}}
\date{}
\pgfplotsset{compat=1.16}
\begin{document}
\maketitle
\begin{abstract}
A full characterization of the boundedness of Laplace--Carleson embeddings on $L^\infty$ is provided, in terms of the Carleson intensity of the respective measure and of a suitable weighted Berezin transform of the measure. Moreover, boundedness results, and in some cases full characterizations of boundedness, are proved for a large class of  Orlicz spaces. These findings are crucial for characterizing admissibility of control operators for linear diagonal semigroup systems in a variety of contexts. A particular focus is laid on essentially bounded inputs.
\end{abstract}


{\bf 2000 Subject Classification.} 93D25, 93B28, 47D06, 46E15.

\section{Introduction}

This paper deals with so-called \textit{Laplace--Carleson embeddings}, which are maps of the form $\mathcal{L}\colon Z\to \mathrm{L}^{q}(\CC_+,d\mu)$ given by
\[\LL f(z):= \int_0^\infty e^{-zt} f(t) dt,\qquad z\in \CC_{+}.\]
Here $Z$ is a function space on $(0,\infty)$ and by $\CC_+$ we denote the open right half-plane of $\CC$. Thus Laplace--Carleson embedding are Carleson embeddings induced by the Laplace transform. In prior works \cite{jpp13,jpp12,Rydhe20}, Laplace--Carleson embeddings have been investigated for $Z=\mathrm{L}^{p}$, $1\le p<\infty$. A detailed review on known results is given in Section \ref{sec3}. In this article our focus lies on $Z=\mathrm{L}^\infty$ and $Z=\mathrm{L}^\Phi$, where $\mathrm{L}^\Phi$ denotes an Orlicz space. Beside sufficient and necessary conditions for the boundedness of the Laplace--Carleson embedding, in the case $q\ge 2$ we show that  the boundedness of $\mathcal{L}\colon \mathrm{L}^\infty\to \mathrm{L}^{q}(\CC_+,d\mu)$ implies the boundedness of $\mathcal{L}\colon \mathrm{L}^\Phi\to \mathrm{L}^{q}(\CC_+,d\mu)$ for some Orlicz space $\mathrm{L}^\Phi$, and we investigate the Laplace-Carleson embedding for $\mathrm{L}^\Phi$- and $\mathrm{L}^\infty$-functions supported on $(0,\tau_0)$ for some $\tau_0>0$.

The motivation for extending the results obtained in \cite{jpp13,jpp12,Rydhe20} is the applicability to admissible control operators for diagonal semigroups. The input-to-state map $u\mapsto x(t_{0})$ of the standard linear control system
\begin{equation}\label{eq:ABsystem}
	\dot x(t)=Ax(t)+Bu(t), \qquad x(0)=0, \quad t \ge 0,
\end{equation}
is given by 
\[\Theta(u)=\int_{0}^{t_{0}}T(t_{0}-s)Bu(s)\mathrm{d}s.\]
Here  $(T(t))_{t\ge 0}$ is a $C_0$-semigroup of bounded linear operators on a Banach space $X$ generated by $A$ and $B$ is a linear bounded operator from a Banach space $U$, the input space, to $X_{-1}$, an extrapolation space of $X$. Here the semigroup operators are identified with their unique extension to $X_{-1}$. The precise definition of this extrapolation space can be found in Section \ref{sec2}.

Therefore, maps of the type $\Theta$ are fundamental to understanding well-posedness and stability of linear  control systems. We remark, that if the semigroup $(T(t))_{t\ge 0}$ is diagonal with respect to some $q$-Riesz basis and $U=\CC$, then the mapping $\Theta$ is given by a Laplace-Carleson embedding. We recall that $B$ is a {\it $Z$-admissible \textcolor{black}{control} operator} (for $A$) if $\Theta:Z(0,t_{0};U)\to X_{-1}$ is well-defined, and bounded as a map from $Z(0,t_{0};U)$ to $X$. Again $Z(0,t_0;U)$ is a Banach space of $U$-valued functions on $(0,t_0)$. The case $Z=\mathrm{L}^{p}$, in particular $p=2$, is commonly studied in the literature, see \cite{jp} and the references therein. The  case $p=\infty$ is of great importance as it corresponds to bounded inputs, but it is also the most difficult to analyse, and our results here answer questions that have been implicitly open for several years. The Orlicz space case $Z=\mathrm{L}^\Phi$ for a certain class of Young functions $\Phi$\textcolor{black}{, referred to in this paper as $N$-functions, }was shown in \cite{JNPS} to play a key role in the analysis of input-to-state stability. As an application of the abstract Laplace-Carleson embedding results obtained in Section \ref{sec3} we are able to answer a question posed in \cite{JNPS} in the case of diagonal semigroups. We show that $\mathrm{L}^\infty$-admissibility implies $\mathrm{L}^\Phi$-admissibility for some  Orlicz space $\mathrm{L}^\Phi$, with $\Phi$ \textcolor{black}{ an $N$-function}\sout{ belonging to the class mentioned above}. Further, we obtain new results in this direction for not necessarily diagonal semigroup. We show that for left-invertible semigroups on Hilbert spaces $\mathrm{L}^\infty$-admissibility even implies $\mathrm{L}^2$-admissibility.

The organization of the paper is as follows: 
Section \ref{sec3} is devoted to Laplace--Carleson embeddings, while
in Section \ref{sec2} we present several results formulated in the language of admissible operators.

\subsection{Notation}
For the rest of the paper, $A$ will always denote the generator of a $C_{0}$-semigroup $(T(t))_{t\ge0}$. Further assumptions on the semigroup may be imposed in the respective sections. The spaces $X$ and $U$ will refer to general complex Banach spaces, unless specified otherwise. The space of bounded linear operators from $U$ to $X$ will be denoted by ${L}(U,X)$.

Following the the notation in \cite[Chapter~1.3]{Rao} and  \cite[Chapter~4.8]{BennetSharpley}, a Young function is a function $\Phi: [0, \infty)  \rightarrow \overline{\RR}_+$ of the form
\begin{equation}   \label{def:young}
   \Phi(t) = \int_0^t \phi(s) ds   \quad (t \ge 0),
\end{equation}
where $\phi: [0, \infty)  \rightarrow \overline{\RR}_+$ is a nondecreasing, left-continuous function which is neither constantly $0$ nor constantly $+\infty$. 

Given a Young function $\Phi$ with left derivative $\phi$ as in (\ref{def:young}), the complementary Young function $\Phi^c$ is given by
\begin{equation}   \label{def:compl}
   \Phi^c(t) = \int_0^t \phi^c(s) ds   \quad (t \ge 0),
\end{equation}
where $\psi^c: [0, \infty)  \rightarrow \overline{\RR}_+$ is the generalized inverse of $\phi$.
For an interval $I \subseteq \RR$, we can then define the Orlicz space $\mathrm{L}^{\Phi}(I;U)$ as the space of $U$-valued measurable functions $f:I\to U$
such that $\Phi(k^{-1}\|f(\cdot)\|_U)$ is integrable for some $k>0$. This space is a Banach function space, equipped with the norm
\begin{equation}\label{eq:Orlicznorm}
\|f\|_{\mathrm{L}^{\Phi}}=\inf\{k>0\colon \int_{I}\Phi(k^{-1}\|f(s)\|_U) \, \mathrm{d}s\leq1\}.
\end{equation}
For $U = \CC$, we write simply $\mathrm{L}^{\Phi}(I)$. Note that for any Young function $\Phi$,
\begin{equation}  \label{trivialincl}
 L^1(I) \cap L^\infty(I) \subseteq L^\Phi(I).
\end{equation}

It is a consequence of Young's inequality that a H\"older inequality for Orlicz spaces holds \cite[Chapter~1.3]{Rao},
\begin{equation}   \label{holderyoung}
    \| fg \|_{L^1(I)} \le 2 \| f \|_{L^\Phi(I)} \| g\|_{L^{\Phi^c}(I)}.
\end{equation}

We follow the notation of \cite{Rao} in calling an $\RR$-valued Young function $\Phi$ with the property that
\begin{equation}   \label{def:nfunction}
        \lim_{x\to0}\frac{\Phi(x)}{x}=\lim_{x\to\infty}\frac{x}{\Phi(x)}=0
\end{equation}
an $N$-function.  The complementary function of an N-function is again an N-function. This is the appropriate class for the applications of Orlicz spaces in the theory of infinite-dimensional linear systems introduced in \cite{JNPS}. Note however that this notation is not uniform in the literature,
  \cite{JNPS}  defines the class of Young functions as the class of $N$-functions in our notation.

 For more details on Orlicz spaces, we refer the reader to textbooks such as \cite{Rao}, \cite[Chapter~4.8]{BennetSharpley} or the appendix of \cite{JNPS}.

\section{Laplace--Carleson embeddings}
\label{sec3}
Let $\mu$ be a positive regular Borel measure on the complex right half-plane $\CC_+:=\{z=x+iy\mid y\in\mathbb{R},x>0\}$. We also use the shifted right half-planes $\CC_\alpha:=\{z=x+iy\mid y\in\mathbb{R},x>\alpha\}$. In this section, we only consider scalar-valued Orlicz spaces $Z$ on the interval $(0,\infty)$, that is $Z=Z(0,\infty;\CC)$ in our notation above. We will omit the reference to the interval here for the sake of brevity. Formally, what we mean by a \textit{Laplace--Carleson embedding} is a map of the form $\mathcal{L}\colon Z\to \mathrm{L}^{q}(\CC_+,d\mu)$ given by
\[\LL f(z):= \int_0^\infty e^{-zt} f(t) dt,\qquad z\in \CC_{+}.\]
Since convergence of a sequence in $Z$ implies pointwise convergence of the corresponding sequence of Laplace transforms,  any set inclusion of the form $\mathcal{L}Z\subseteq \mathrm{L}^{q}(\CC_+,d\mu)$ is automatically continuous by the closed graph theorem. 

The \emph{Hardy space} $\mathrm{H}^p(\CC_+)$ consists of all analytic functions $F\colon\CC_+\to\CC$ for which
\[
	\|F\|_{\mathrm{H}^p(\CC_+)}^p:=\sup_{\epsilon >0}\int_{y\in\RR}|F(\epsilon+iy)|^p\, dy<\infty.
\]
For the shifted half-plane $\CC_{\alpha}$, we have accordingly the Hardy space
$\mathrm{H}^p(\CC_{\alpha})$ of all analytic functions on $\CC_{\alpha}$ such that
$$
	\|F\|_{\mathrm{H}^p(\CC_{\alpha})}^p =\sup_{\epsilon >0}\int_{y\in\RR}|F(\epsilon+ \alpha + iy)|^p\, dy<\infty.
$$
For $F\in \mathrm{H}^p(\CC_+)$ and $F_\epsilon(iy):=F(\epsilon+iy)$, the limit $bF(iy)\:=\lim_{\epsilon\to 0^+}F_\epsilon(iy)$ exists for Lebesgue a.e. $y$. Moreover, $F_\epsilon\to bF$ in $\mathrm{L}^{p}(i\RR)$, provided that $p<\infty$. This makes $\mathrm{H}^p(\CC_{+})$ isometrically isomorphic to a closed subspace of $\mathrm{L}^{p}(i\RR)$. A good reference on Hardy spaces is \cite[Chapter~II]{garnett}.

For $\lambda\in\CC_{+}$ and $t>0$, let $k_\lambda(t)=\frac{1}{2\pi}\exp(-\bar\lambda t)$. Note that $\|k_\lambda\|_{\mathrm{L}^{p}}^p=\frac{1}{p(2\pi)^p\re\lambda}$. The so-called \emph{reproducing kernel} for $H^p(\CC_{+})$ is the analytic function
\[K_\lambda\colon z\mapsto \mathcal{L}k_\lambda(z)=\frac{1}{2\pi}\frac{1}{z+\bar\lambda},\]
defined at least for $\re z \ge 0$. If $1\le p<\infty$ and $F\in \mathrm{H}^p(\CC_+)$, then
\begin{equation}\label{eq:ReproducingFormula}
	F(\lambda)=\int_{y\in\RR}F(iy)\overline{K_\lambda(iy)}\, dy,
\end{equation}
which follows essentially from Cauchy's theorem.

\subsection{Laplace--Carleson embeddings and Carleson intensities}\label{sec21}

The {\it Carleson square} associated to an interval $I\subset i\mathbb{R}$ is the set
\[
Q_I:=\left\{ z=x+iy \in \CC_+\mid iy\in I,0<x\le |I| \right\},
\]
the right half of the Carleson square is
\begin{equation}   \label{eq:righthalf}
T_I:=\left\{ z=x+iy \in \CC_+\mid iy\in I,   |I|/2<x\le |I| \right\}.
\end{equation}
These are related to reproducing kernels by the fact that if $ \lambda$ is the centre of $Q_I$, so that in particular $\Re \lambda  = |I|/2$, then
\[
\frac{1}{\sqrt{10}\pi|I|}\le|K_\lambda(z)|\le \frac{1}{\pi|I|} \quad  \text{ for }  z \in Q_I.
\]
With $p'$ denoting the H\"older conjugate of $p\in[1,\infty]$, the above inequalities imply that if $\mathcal{L}\colon \mathrm{L}^{p}\to \mathrm{L}^{q}(\CC_+,d\mu)$ is bounded, then
\begin{equation}\label{eq:NecessaryIntensity}
	\mu(Q_I)\lesssim |I|^{q/p'} \quad\textnormal{for all intervals}\quad I\subset i\mathbb{R},
\end{equation}
see \cite[Proposition~3.1]{jpp13}. In the case $q > p'$, $Q_I$ can equivalently be replaced by $T_I$ in 
(\ref{eq:NecessaryIntensity}). It is a remarkable fact that in a variety of situations, condition \eqref{eq:NecessaryIntensity} is also sufficient for $\mathcal{L}\colon \mathrm{L}^{p}\to \mathrm{L}^{q}(\CC_+,d\mu)$ to be bounded. For $1\le p\le 2$, and $p'\le q <\infty$ (this corresponds to the region I in Figure \ref{figure:KnownAndNewResults}), $\mathcal{L}\colon \mathrm{L}^{p}\to \mathrm{L}^{q}(\CC_+,d\mu)$ if and only if \eqref{eq:NecessaryIntensity} holds, see \cite[Theorem~3.2]{jpp13}. 
In \cite[Theorem~1.1]{Rydhe20}, this result was extended to $2<p\le q<\infty$ (region II in Figure \ref{figure:KnownAndNewResults}). 
For $2\le q < p< \infty$ (region III in Figure \ref{figure:KnownAndNewResults}), \eqref{eq:NecessaryIntensity} is sufficient if $\mu$ has support on a vertical strip, but not if $\mu$ has support on a sector, see \cite[Theorem~3.6 and Theorem~3.5]{jpp13}. The thick line in Figure~\ref{figure:KnownAndNewResults} corresponds to the hypothesis of Theorem~\ref{thm:CharacterizationLpToLq} below. This new result characterizes the class of $\mu$ such that $\LL\colon \mathrm{L}^{\infty}\to \mathrm{L}^{q}(\CC_+,d\mu)$ for $q\ge 2$. 

\begin{figure}
	\begin{center}
		\begin{tikzpicture}[scale=0.7]
			\begin{axis}[xmin=0, xmax=1.1, ymin=0, ymax=1.1, 
				axis lines = left,
				x label style={at={(axis description cs:1,0)},anchor=west},
				y label style={at={(axis description cs:0,1)},rotate=270,anchor=east},
				xlabel={$\frac{1}{p}$},
				ylabel={$\frac{1}{q}$},
				scaled x ticks=false,
				scaled y ticks=false,
				xtick={0.5,1},
				xticklabels={$\frac{1}{2}$,$1$},
				ytick={0.5,1},
				yticklabels={$\frac{1}{2}$,$1$},
				width=0.618\textwidth,
				height=0.618\textwidth
				]
				\addplot[black] coordinates
				{(0,0) (0.5,0.5)};
				\addplot[black] coordinates
				{(0.5,0.5) (1,0)};
				\addplot[black] coordinates
				{(0,1) (1,1)};
				\addplot[black] coordinates
				{(1,1) (1,0)};
				\addplot[black] coordinates
				{(0.5,0.5) (.5,0)};
				\addplot[black] coordinates
				{(0,0.5) (.5,.5)};
				\addplot[black,style = {line width = 6pt}] coordinates
				{(0,0) (0,.5)};
				\node[black] at (axis cs:0.667,0.167){I};
				\node[black] at (axis cs:0.333,0.167){II};
				\node[black] at (axis cs:0.167,0.333){III};
			\end{axis}
		\end{tikzpicture}
	\end{center}
	\caption{Relation between condition \eqref{eq:NecessaryIntensity} and the boundedness of  $\mathcal{L}\colon \mathrm{L}^{p}\to \mathrm{L}^{q}(\CC_+,d\mu)$. If $1\le p\le 2$, $p'\le q <\infty$ (region I), or $2<p\le q<\infty$ (region II), then \eqref{eq:NecessaryIntensity} is necessary and sufficient for the embedding to be bounded. If $2\le q < p\le \infty$ (region III), then \eqref{eq:NecessaryIntensity} is necessary and sufficient under the additional assumption that $\mu$ has support on a vertical strip. For general measures, \eqref{eq:NecessaryIntensity} is necessary but not sufficient. The bold edge to the left corresponds to the hypothesis of Theorem~\ref{thm:CharacterizationLpToLq}.}\label{figure:KnownAndNewResults}
\end{figure}
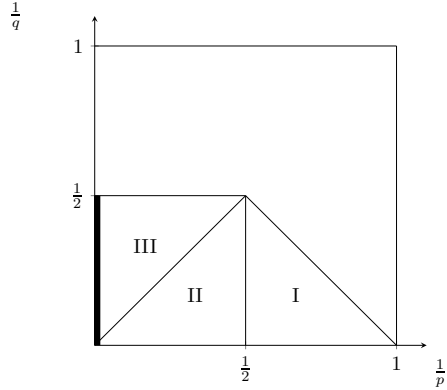

Motivated by the significance of \eqref{eq:NecessaryIntensity}, we introduce the concept \textit{$\alpha$-Carleson intensity}.

\begin{definition}
Let $\mu$ be a positive regular Borel measure on $\CC_+$ and $\alpha>0$. Then the  {\rm{$\alpha$-Carleson intensity}} $\mathcal{C}_\alpha[\mu]$ is given by
\[
\mathcal{C}_\alpha[\mu]=\sup_{\substack{I\subset i\mathbb{R}\\I\textnormal{ interval}}}\frac{\mu(Q_I)}{|I|^\alpha}.
\]
For $t >0$, the    {\rm{$\alpha$-Carleson intensity at scale $t$}}, $\mathcal{C}_\alpha[\mu](t)$, is given by
$$
\mathcal{C}_\alpha[\mu](t)=\sup_{\substack{I\subset i\mathbb{R}\\I\textnormal{ interval, } |I| =t}}\frac{\mu(Q_I)}{t^\alpha}.
$$
\end{definition}

Obviously, \eqref{eq:NecessaryIntensity} holds if and only if $\mathcal{C}_{q/p'}[\mu]<\infty$.

Measures supported on vertical strips will play an important role in the investigation below. The next definition is essentially a notational convention that will be used henceforth.
\begin{definition} Let $\mu$ be a positive regular Borel measure on $\CC_+$. For each $n\in\ZZ$, consider the \emph{dyadic strip}
	\[
	S_n:	=\left\{x+iy\mid y\in\mathbb{R},2^n\le x<2^{n+1}\right\},
	\]
	and define the measure $\mu_n$ on $\CC_{+}$ by $\mu_n\colon E\mapsto \mu(E\cap S_n)$.
\end{definition}

If $2\le q <p\le\infty$, and $\mu$ is supported on a vertical strip, then $\mathcal{L}\colon \mathrm{L}^{p}\to \mathrm{L}^{q}(\CC_+,d\mu)$ if and only if \eqref{eq:NecessaryIntensity} holds. The main technical difficulty in this paper is to characterize $\mathcal{L}\colon \mathrm{L}^{p}\to \mathrm{L}^{q}(\CC_+,d\mu)$ in terms of $\mu$, without imposing any further conditions on the support of $\mu$. In the following, we address this for the case $p=\infty$.

We require one further piece of notation. Recall that for $\alpha > -1$ and a positive regular Borel measure  $\mu$ on $\CC_+$, its Berezin transform $B_\alpha \mu$ is defined as
$$
   B_\alpha 
   \mu(z) = \int_{\CC^+}    \frac{ (\Re z)^{2+ \alpha}}{|w + \bar z|^{4 + 2 \alpha}}
   	 (\Re w)^\alpha d\mu(w)   \quad(z \in \CC_+).
$$
As an auxiliary expression, we write
$$
\widetilde B_\alpha 
\mu(z) = \int_{\CC^+}    \frac{ (\Re z)^{2+2 \alpha}}{|w + \bar z|^{4 + 2 \alpha}} d\mu(w)   \quad(z \in \CC_+).
$$

One verifies easily that for $ \alpha > -1$, $4 + 2 \alpha > \beta >0$,
\begin{equation}   \label{BerCarl}
	\sup_{t >0}    t^{2-\beta}    \sup_{\Re z =t}  \widetilde B_\alpha \mu(z) < \infty
\end{equation}
is an equivalent statement of the Carleson intensity condition $\carl_{\beta}[\mu] < \infty$.

Here is our main result of this section.

\begin{theorem}\label{thm:CharacterizationLpToLq}
	Let $\mu$ be a positive regular Borel measure on $\CC_+$, $\alpha > -1$, and $2 \le q$.
     
     Then the following are equivalent:
	\begin{enumerate}
	\item The Carleson-Laplace embedding 
	\begin{equation}  \label{Boundedness} 
	\mathcal{L}\colon \mathrm{L}^{\infty}(0,\infty)\to \mathrm{L}^{q}(\CC_+,d\mu) \text{ is bounded. }
	\end{equation}
	\item 
	\begin{equation}\label{eq:NecessarySummability}
		\sum_{n\in\ZZ}\mathcal{C}_{q}[\mu_n]<\infty.
	\end{equation} 
\item 
\begin{equation}\label{eq:Intensity}
\int_0^\infty    \frac{1}{t} \mathcal{C}_{q}[\mu](t) dt<\infty.
\end{equation} 
\end{enumerate}

If $\alpha > -1$ and $q < 4 + 2 \alpha$, then (1.) - (3.) is equivalent to
\begin{enumerate}
\setcounter{enumi}{3}
	\item    
\begin{equation}\label{eq:Berezin}
	\int_0^\infty    t^{1-q}  \sup_{\Re z =t } \widetilde B_\alpha \mu(z) dt < \infty.
\end{equation} 

\end{enumerate}

If $\alpha > -1$ and $q < 4 + \alpha$, then (1.) - (3.) is equivalent to
\begin{enumerate}
\setcounter{enumi}{4}
\item
\begin{equation}     \label{eq:realBerezin}
	\int_0^\infty    t^{1-q}  \sup_{\Re z =t } B_\alpha \mu(z) dt < \infty.
\end{equation}

\end{enumerate}

	Furthermore, the above sum respectively integrals  are comparable to $$\|\mathcal{L}\colon \mathrm{L}^{\infty}(0,\infty)\to \mathrm{L}^{q}(\CC_+,d\mu)\|^q,$$ with implied constants only depending on $q$ and $\alpha$.
	
\end{theorem}

We begin with the equivalence of (\ref{Boundedness}) and  \eqref{eq:NecessarySummability}. The implication of $(\ref{Boundedness}) \Rightarrow \eqref{eq:NecessarySummability} $ follows from the more general Theorem~\ref{thm:NecessaryLpToLq} below. The reverse implication follows from a different generalization, Theorem~\ref{thm:SufficiencyLPhitoLq}. Before proving these results, we need the following simple lemma. 

\begin{lemma}\label{lemma:GeometricPropertiesOfIntesities}
Let $\alpha\ge 1$.
	\begin{enumerate}[$(i)$]
		\item There exists an interval $I\subset i\mathbb{R}$ such that $|I|=2^{n+1}$ and
		\[
			\mathcal{C}_\alpha[\mu_n]\le 2^{\alpha + 1}\frac{\mu_n(Q_I)}{|I|^\alpha}.
		\]
		\item If $\beta\ge 1$, then
		\[
			\mathcal{C}_\alpha[\mu_n]\le 2^{\beta+n(\beta-\alpha)}\mathcal{C}_\beta[\mu_n]\le 2^{\alpha+\beta}\mathcal{C}_\alpha[\mu_n],
		\]
		i.e. $\mathcal{C}_\alpha[\mu_n]\approx 2^{n(\beta-\alpha)}\mathcal{C}_\beta[\mu_n]$, where the constants of comparison depend only on $\alpha$ and $\beta$.
		\item If one defines the shifted measure $\tilde\mu_n\colon E\mapsto \mu_n(E+2^{n-1})$, then
		\[
			\mathcal{C}_\alpha[\tilde\mu_n]\le 2^\alpha\mathcal{C}_\alpha[\mu_n].
		\]
	\end{enumerate}
\end{lemma}

\begin{proof} To prove $(i)$, introduce the auxiliary quantity
\[
C_{\alpha}[\mu_n](2^{-(n+1)}) =
	\widetilde{\mathcal{C}_\alpha}[\mu_n]=\sup_{|I|=2^{n+1}}\frac{\mu_n(Q_I)}{|I|^\alpha}.
\]
If $|I|\ge 2^n$, then there exists a finite collection of intervals $\{J_k\}_{k=1}^N$, where $N\le 2^{-n}|I|$, each $|J_k|=2^{n+1}$, and $I\subseteq\bigcup_{k=1}^NJ_k$. Since also $Q_I\cap S_n\subset \bigcup_{k=1}^NQ_{J_k}$,
\begin{eqnarray*}
		\mu_n(Q_I)&\le& \sum_{k=1}^N\mu_n(Q_{J_k})\le \sum_{k=1}^N\widetilde{\mathcal{C}_\alpha}[\mu_n]\left(2^{n+1}\right)^{\alpha}\le
	2^\alpha\sum_{k=1}^N\widetilde{\mathcal{C}_\alpha}[\mu_n]\left(\frac{|I|}{N}\right)^{\alpha}\\
	&\le& 2^\alpha|I|^\alpha\widetilde{\mathcal{C}_{\alpha}}[\mu_n].
\end{eqnarray*}
For smaller intervals, $\mu_n(Q_I)=0$. From this, $\mathcal{C}_\alpha[\mu_n]\le 2^\alpha \widetilde{\mathcal{C}_{\alpha}}[\mu_n]$, and since there clearly exists $I$ with $|I|=2^{n+1}$ such that $\widetilde{\mathcal{C}_{\alpha}}[\mu_n]\le 2\frac{\mu(Q_I)}{|I|^\alpha}$, $(i)$ follows.

For the proof of $(ii)$, it is immediate from the definition that $\widetilde{\mathcal{C}_{\alpha}}[\mu_n]=2^{(n+1)(\beta-\alpha)}\widetilde{\mathcal{C}_{\beta}}[\mu_n]$. Since $\widetilde{\mathcal{C}_{\beta}}[\mu_n]\le \mathcal{C}_{\beta}[\mu_n]$, and we just proved that $\mathcal{C}_\alpha[\mu_n]\le 2^\alpha \widetilde{\mathcal{C}_{\alpha}}[\mu_n]$, this establishes the first inequality in $(ii)$. The second inequality follows by interchanging $\alpha$ and $\beta$.

To prove $(iii)$, note that $\mu_n(Q_I+2^{n-1})=0$ when $|I|<2^{n-1}$, whereas if $|I|\ge 2^{n-1}$, then $Q_I+2^{n-1}\subseteq Q_{2I}$. 
\end{proof}

The necessity of \eqref{eq:NecessarySummability} for the boundedness of the Laplace--Carleson embedding
(\ref{Boundedness}) extends to $1\le q<\infty$:

\begin{theorem}\label{thm:NecessaryLpToLq}
If $1\le q<\infty$, then
	\begin{equation*}
		\sum_{n\in\ZZ}\mathcal{C}_{q}[\mu_n]\lesssim \|\mathcal{L}\colon \mathrm{L}^{\infty}\to \mathrm{L}^{q}(\CC_+,d\mu)\|^q.
	\end{equation*} 
\end{theorem}

\begin{proof}
For $n\in\ZZ$, choose $I_n$ with $|I_n|=2^{n+1}$, and $\mathcal{C}_q[\mu_n]\le 2^{q+1}\frac{\mu_n(T_n)}{|I_n|^q}$, where $T_n$ denotes the right-hand half of the square $Q_{I_n}$. With $ic_n$ denoting the mid-point of $I_n$, define $f_n(t)=\chi_{(2^{-n-1},2^{-n}]}(t)e^{ic_nt}$, and $F_n=\LL f_n$. The proof now proceeds through three steps.

\paragraph{Step 1:}
We first show that there exist positive real constants $c$ and $C$ such that:
	\begin{enumerate}[$(i)$]
		\item If $n\in\ZZ$ and $z\in T_n$, then $|F_n(z)|\ge c 2^{-n}$.
		\item If $m,n\in\ZZ$ and $z\in T_n$, then $|F_m(z)|\le C 2^{-n-|n-m|}$.
	\end{enumerate}
\proof
$(i)$ If $z = x + i y \in T_n$ and $2^{-n-1}\le t\le 2^{-n}$, then $|t(y-c_n)|\le 1$. Hence, 
\[
|F_n(z)|\ge \Re F_n(z)=\int_{t=2^{-n-1}}^{2^{-n}}e^{-tx}\cos\left(t\left(y-c_n\right)\right)\, dt\ge e^{-2}\cos(1)2^{-n-1}.
\]

$(ii)$ By the triangle inequality,
\[
|F_m(z)|\le \int_{t=2^{-m-1}}^{2^{-m}}e^{-xt}\, d  t.
\]
Since the above integral is less than $2^{-m}=2^{-n-(m-n)}$, our inequality is immediate for $m\ge n$. For $m<n$, we use instead that
\begin{eqnarray*}
|F_m(z)|&\le& \int_{t=2^{-m-1}}^{\infty}e^{-xt}\, d  t \\
&=& \frac{e^{-2^{-m-1}x}}{x}\\
&\le& \frac{e^{-2^{n-m-1}}}{2^{n}}=2^{-n}2^{n-m}e^{-2^{n-m-1}}2^{m-n}.
\end{eqnarray*}
Since $2ae^{-a}$ is bounded for $a\ge 0$, the conclusion follows.
\qed

\paragraph{Step 2:} With $c$ and $C$ as in Step~1, choose an integer $N$ such that $C2^{3-N}\le c$. For $k\in\{1,2,\ldots,N\}$, define $g_k=\sum_{m\in\ZZ}f_{mN+k}$, and $G_k=\LL g_k$. We now show that if $n\in\ZZ$ and $z\in T_{nN+k}$, then $|G_k(z)|\ge \frac{1}{2}|F_{nN+k}(z)|$.
\proof
By the properties in Step~1,
\begin{multline*}
	|G_k(z)-F_{nN+k}(z)|\le \sum_{\substack{m\in\ZZ\\m\ne n}}|F_{mN+k}(z)|\le C2^{-nN-k}\sum_{\substack{m\in\ZZ\\m\ne n}}2^{-N|n-m|}\\
	=\frac{C2^{-nN-k+1-N}}{1-2^{-N}}
	\le
	C2^{-nN-k+2-N}
	\le \frac{c}{2}2^{-nN-k}\le\frac{1}{2}|F_{nN+k}(z)|.
\end{multline*}
The result now follows from the reverse triangle inequality.
\qed

\paragraph{Step 3:} We are now ready to complete the proof of Theorem~\ref{thm:NecessaryLpToLq}. With $N$ as above,
\begin{equation*}
	\sum_{n\in\ZZ}\mathcal{C}_q[\mu_n]
	\le
	2^{q+1}\sum_{n\in\ZZ}\frac{\mu_n(T_n)}{2^{(n+1)q}}
	= 
	2^{q+1}\sum_{k=1}^N\sum_{n\in\ZZ}2^{-(nN+k+1)q}\mu_n(T_{nN+k}).	
\end{equation*}
According to the previous steps,
\[
2^{-(nN+k+1)q}\lesssim |F_{nN+k}(z)|^q\lesssim |G_{k}(z)|^q,
\]
whenever $z\in T_{nN+k}$. Hence,
\begin{equation*}
	2^{-(nN+k+1)q}\mu_n(T_{nN+k})\lesssim \int_{T_{nN+k}}|G_k|^q\, d \mu ,
\end{equation*}
and
\begin{equation*}
	\sum_{n\in\ZZ}\mathcal{C}_q[\mu_n]
	\lesssim 
	\sum_{k=1}^N\sum_{n\in\ZZ}\int_{T_{nN+k}}|G_k|^q\, d \mu
	\le 
	\sum_{k=1}^N\int_{\CC_+}|G_k|^q\, d \mu.
\end{equation*}
Since $\|g_k\|_{\mathrm{L}^{\infty}}=1$, $\sum_{n\in\ZZ}\mathcal{C}_q[\mu_n]\lesssim \|\LL\colon \mathrm{L}^{\infty}\to \mathrm{L}^{q}(\CC_+,d\mu)\|^q$, with implied constants only depending on $q$.
\end{proof}

The sufficiency of \eqref{eq:NecessarySummability} for the boundedness of the Laplace-Carleson Embedding (\ref{Boundedness}) can be extended to the situation where $\mathrm{L}^{p}$ is replaced by certain Orlicz spaces $\mathrm{L}^{\Phi}$.
For Orlicz spaces, we have the following variant of sufficiency of \eqref{eq:NecessarySummability}, generalizing the sufficiency part of Theorem \ref{thm:CharacterizationLpToLq}.

\begin{theorem}\label{thm:SufficiencyLPhitoLq}
	
	Let  $ 2 \le q < \infty$. Let $\Phi$ be a Young function of the form $\Phi(t)=\tilde\Phi(t^{q'})$, where $\tilde \Phi$ is another Young function. Then it holds that
	\begin{equation}\label{eq:SufficientLPhiToLq}
		\|\mathcal{L}\colon \mathrm{L}^{\Phi}(0,\infty)\to \mathrm{L}^{q}(\CC_{+},d\mu)\|^q\lesssim
		\sum_n \left(2^{n}\|\exp^{-q'2^{n-1}}\|_{\mathrm{L}^{\tilde \Phi^c}}\right)^{q-1}\mathcal{C}_{q}[\mu_n].
	\end{equation}
\end{theorem}

\begin{remark}   \label{stabilityYoung}
	It is clear that if $\tilde\Phi$ is a Young function, then so is $\Phi\colon t\mapsto \tilde\Phi(t^{q'})$. The corresponding statement is true if $\tilde \Phi$ is an $N$-function. The converse is not true. The present construction ensures that $\Phi$ increases ``not too slowly'' relative to $t\mapsto t^{q'}$.
\end{remark}

\begin{proof}[Proof of Theorem \ref{thm:SufficiencyLPhitoLq}]
To prove \eqref{eq:SufficientLPhiToLq}, we need two main tools. The first is the classical Hausdorff--Young theorem: Given $1\le p \le 2$, the Fourier transform is a bounded map from $\mathrm{L}^{p}(\mathbb{R})$ to $\mathrm{L}^{p'}(\mathbb{R})$. This readily implies boundedness of $\LL\colon \mathrm{L}^{p}(0,\infty)\to \mathrm{H}^{p'}(\CC_{+})$. The second tool is the Carleson embedding theorem, e.g. \cite[Theorem~II.3.9]{garnett}, which states that $\|\mathrm{H}^q(\CC_{+})\hookrightarrow \mathrm{L}^{q}(\CC_+,d\mu)\|^q$ is comparable to $\mathcal{C}_1[\mu]$.

Let $f\colon(0,\infty)\to\CC$ be such that $F=\LL f$ is well-defined as an analytic function on $\CC_+$. The following calculations will yield that this is always the case when $f\in \mathrm{L}^{\Phi}$. 

It holds that
\begin{eqnarray*}
	\int_{\CC_+}|F|^q\, d\mu 
	&=&
	\sum_{n\in\ZZ}\int_{\CC_+}|F|^q\, d\mu_n\\
	&=&
	\sum_{n\in\ZZ}\int_{\CC_+}|F(z+2^{n-1})|^q\, d\tilde \mu_n(z)\\
	&=&
	\sum_{n\in\ZZ}\int_{\CC_+}|\LL(f\exp^{-2^{n-1}})|^q\, d\tilde \mu_n,
\end{eqnarray*}
where $\tilde\mu_n$ is the shifted measure $E\mapsto \mu_n(E+2^{n-1})$ appearing in Lemma~\ref{lemma:GeometricPropertiesOfIntesities}. In combination with Carleson's theorem and the Hausdorff--Young theorem, we obtain
\begin{eqnarray*}
	\int_{\CC_+}|\LL(f\exp^{-2^{n-1}})|^q\, d\tilde \mu_n
	&\lesssim&
	\mathcal{C}_1[\tilde\mu_n]\|\LL(f\exp^{-2^{n-1}})\|_{\mathrm{H}^q}^q
	\\
	&\lesssim&
	\mathcal{C}_1[\tilde\mu_n]\,\|f\exp^{-2^{n-1}}\|_{\mathrm{L}^{q'}}^q\\
	&=&
	\mathcal{C}_1[\tilde\mu_n]\,\||f|^{q'}\exp^{-q'2^{n-1}}\|_{\mathrm{L}^{1}}^{q/q'}.	
\end{eqnarray*}
Appealing to Lemma~\ref{lemma:GeometricPropertiesOfIntesities}, $\mathcal{C}_1[\tilde\mu_n]\lesssim 2^{n(q-1)}\mathcal{C}_q[\mu_n]$. We now apply H\"older's inequality for Orlicz spaces
 (\ref{holderyoung}) to control
\begin{eqnarray*}
	\left\||f|^{q'}\exp^{-q'2^{n-1}}\right\|_{\mathrm{L}^{1}} \le 2
	\left\||f|^{q'}\right\|_{\mathrm{L}^{\tilde \Phi}}\left\|\exp^{-q'2^{n-1}}\right\|_{\mathrm{L}^{\tilde\Phi^c}}.
\end{eqnarray*}
By (\ref{trivialincl}), $\|\exp^{-q'2^{n-1}}\|_{\mathrm{L}^{\tilde\Phi^c}}<\infty$ for any Young function $\tilde\Phi^c$. This shows in particular that $f\exp^{-2^{n-1}}\in \mathrm{L}^{q'}$, so $F(z)=\LL f(z)$ is well-defined for $\re z > 2^{n-1}$. As $n$ is arbitrary, $F\colon \CC_{+}\to \CC$ is well-defined and analytic. It also holds that $\||f|^{q'}\|_{\mathrm{L}^{\tilde \Phi}}^{1/q'}=\|f\|_{\mathrm{L}^{\Phi}}$.  Piecing all of this together, we obtain \eqref{eq:SufficientLPhiToLq}.
\end{proof}

\begin{proof}[Proof of Theorem \ref{thm:CharacterizationLpToLq}]
To prove of the equivalence of  (\ref{Boundedness}) and (\ref{eq:NecessarySummability}), we note that 
 by Theorem~\ref{thm:NecessaryLpToLq},  the boundedness of $\mathcal{L}\colon \mathrm{L}^{\infty}\to \mathrm{L}^{q}(\CC_+,d\mu)$ implies \eqref{eq:NecessarySummability}. To see the reverse implication, we apply Theorem~\ref{thm:SufficiencyLPhitoLq} to the case where $\mathrm{L}^{\Phi}=\mathrm{L}^{\infty}$, in which $\mathrm{L}^{\tilde\Phi^c}=\mathrm{L}^{1}$. 

It remains to show the equivalence of (\ref{eq:NecessarySummability}),   (\ref{eq:Intensity}), (\ref{eq:Berezin}), and (\ref{eq:realBerezin}). 
  The inequalities
   \begin{equation*}
       \carl_q[\mu_n] \le 2^q \carl_q[\mu](t) 
         \lesssim \tilde B_\alpha \mu(z) \quad      \text{ for } 2^{n+1} \le t \le 2^{n+2}, \Re z =t
   \end{equation*}
   show the implication (\ref{eq:Berezin}) $\Rightarrow$ (\ref{eq:Intensity}) $\Rightarrow$
(\ref{eq:NecessarySummability}). 
For the  implication (\ref{eq:NecessarySummability})
$\Rightarrow$ (\ref{eq:Berezin}), choose 
for a given $z \in \CC_+$ an integer $n \in \ZZ$ such that $2^{n} \le \re z < 2^{n+1}$. Let $I$ be the interval in $i \RR$ with  center $\Im z$ and $|I| = 2^{n+1}$, let $Q_I$ be the Carleson square over $I$,  and let $T_I$ be its right half.
For $l \in \ZZ$, let $I_l$ denote the translated interval $I + l  |I| $. We note
$$
 \bigcup_{ l \in \ZZ} Q_{I_{l} } = \bigcup_{ k \le n} S_k
$$
and the standard decay estimate
\begin{equation} \label{eq:decay2}
\frac{ (\Re z)^{2+ 2\alpha} }{|w + \bar z|^{4+ 2 \alpha}}  \approx
\frac{1  }{2^{2n}       (  1 + |l|^2)^{2 +  \alpha}   } 
\quad (w \in Q_{I_l}).
\end{equation}

For $k > n$, let $I^{(k-n)} =2^{k-n} I $ and $I_{ k, l} = I^{(k-n)} + i l 2^{k+1}$ . Here, for $c >0$, $c I$ denotes the interval in $i \RR$ with the same center as $I$ and length $c |I|$. We have
$$
   S_k \subseteq \bigcup_{ l \in \ZZ} Q_{I_{k,l} }
$$
and the decay estimate
\begin{equation} \label{eq:decay3}
\frac{ (\Re z)^{2+ 2\alpha} }{|w + \bar z|^{4+ 2 \alpha}}  \approx
\frac{2^{(n-k)(2 + 2 \alpha)}}{2^{2k} (  1 + |l|^2)^{2 +  \alpha}   } 
\quad (w \in T_{I_{k,l}}).
\end{equation}

Altogether, we find that

\begin{multline*}
	\widetilde{B_\alpha} \mu(z) = \int_{\CC_+}  \frac{ (\Re z)^{2+ 2 \alpha}}{|w + \bar z|^{4+ 2 \alpha}} 
	d\mu(w) \\
	\lesssim 2^{qn-2n}  \carl_q[\mu](2^{n+1})    + 
	\sum_{k > n}  2^{(2+ 2 \alpha)n-(4+ 2 \alpha)k + kq}  \carl_q[\mu_k](2^{k+1}) \\
	\le  2^{-2n} \left( \sum_{k \le n}  2^{n-k}    2^{kq}   \carl_q[\mu_k](2^{k+1})    +  \sum_{k > n} 2^{(4+ 2 \alpha)n-(4+ 2 \alpha)k + kq }    \carl_q[\mu_k](2^{k+1})  \right).
\end{multline*}
Hence
\begin{multline*}
	\int_0^\infty    t^{1-q}  \sup_{\Re z =t } \widetilde{B_\alpha} \mu(z) dt \\
	\lesssim \sum_{n= - \infty}^\infty 2^{n(2-q)} 2^{-2n}    \left(   \sum_{k \le n}2^{n-k}2^{kq}   \carl_q[\mu_k](2^{k+1})    +  \sum_{k > n} 2^{(4+ 2 \alpha)n-(4+ 2 \alpha)k + kq}  \carl_q[\mu_k](2^{k+1})  \right) \\
	= \sum_{n= - \infty}^\infty      \left(   \sum_{k \le n}2^{n-k}2^{(k-n)q}   \carl_q[\mu_k](2^{k+1})    +  \sum_{k > n} 2^{(4 + 2 \alpha) (n-k)+ (k-n)q }  \carl_q[\mu_k](2^{k+1})  \right) \\
	= \sum^\infty_{k= - \infty}        \carl_q[\mu_k](2^{k+1})  \left(   \sum_{n \ge k}             2^{(k-n)(q-1)}       +  \sum_{n < k}    2^{(k-n)(q-4 - 2 \alpha) } \right)   \\
	\lesssim  \sum^\infty_{k= - \infty}        \carl_q[\mu_k],
\end{multline*}
where the implied constants depend only on $4 + 2 \alpha - q$.

Finally, we have to  show the equivalence of (\ref{eq:NecessarySummability}) and
(\ref{eq:realBerezin}). 
Assume that (\ref{eq:NecessarySummability}) holds and that $4 + \alpha > q$.
Note that since $q \ge 2$ and $ \alpha > -1$, we have $q > \alpha -1$.

Again, for a given $z \in \CC_+$ and $l \in \ZZ$, choose $n \in \ZZ$,
$I$ and $I_l$ as  above.

Using the standard estimate
\begin{equation} \label{eq:decay1}
	\frac{ (\Re z)^{2+ \alpha} \Re w^\alpha}{|w + \bar z|^{4+ 2 \alpha}}  \approx
	\frac{2^{(2+ 2 \alpha)n}  }{2^{2n}       (  1 + l^2)^{2 +  \alpha}   } 
	\quad (w \in T_{I_l})
\end{equation}  

 we note  for $k \in \ZZ $, $k \le n$,
 \begin{multline*}
	\int_{S_k}   \frac{ 1}{|w + \bar z|^{4+ 2 \alpha}} 
	d\mu(w)  =   \sum_{l \in \ZZ}    	\int_{S_k \cap Q_{I_l} }   \frac{ 1}{|w + \bar z|^{4+ 2 \alpha}} 
	d\mu(w)   \\
	 \lesssim   	2^{-n (4 + 2 \alpha)}  \sum_{l \in \ZZ}    
(1 + l^2)^{-(4 + 2 \alpha)}  \mu_k(Q_{I_l}) \\
	\le
	2^{-n (4 + 2 \alpha)}  \sum_{l \in \ZZ}    
	(1 + l^2)^{-(2 +  \alpha)} 2^{n-k}     2^{kq}  \mathcal{C}_q [\mu_k] \\
	\lesssim
	2^{-n (3 + 2 \alpha)}      2^{k(q-1)}  \mathcal{C}_q [\mu_k],
	\end{multline*}
	with implied constants only depending on $\alpha$.
	
For $k > n$, let $I^{(k-n)} =2^{k-n} I $ and $I_{ k, l} = I^{(k-n)} + i l 2^{k+1}$  as above. With
$$
   S_k \subseteq \bigcup_{ l \in \ZZ} Q_{I_{k,l} }
$$
and
$$
   | w + \bar z |^2 \approx 2^{2k}  (1 + l^2)   \text{ for } w \in Q_{I_{k,l} },
$$
we obtain
\begin{multline*}
		\int_{S_k}   \frac{ 1}{|w + \bar z|^{4+ 2 \alpha}} 
	d\mu(w)  =   \sum_{l \in \ZZ}    	\int_{S_k \cap Q_{I_{k,l} } }   \frac{ 1}{|w + \bar z|^{4+ 2 \alpha}} 
	d\mu(w)   \\
	\lesssim   	2^{- k(4 + 2 \alpha)}  \sum_{l \in \ZZ}    
	(1 + l^2)^{-(2 +  \alpha)}  \mu_k(Q_{I_l}) \\
	\le
	2^{-k (4 + 2 \alpha)}  \sum_{l \in \ZZ}    
	(1 + l^2)^{-(2 +  \alpha)}     2^{kq}  \mathcal{C}_q [\mu_k] \\
	\lesssim
	2^{-k (4 + 2 \alpha - q)}   \mathcal{C}_q [\mu_k] , \\
\end{multline*}
again with implied constants depending on $\alpha$.
Altogether,
\begin{multline*}
		\int_0^\infty t^{1-q} \sup_{\re z =t} B_\alpha \mu(z) dt \\
	\lesssim  \sum_{n=- \infty} ^\infty  2^{n(2-q)}
	\sup_{ \Re z = 2^n}  2^{n(2+ \alpha)} \sum_{k \in \ZZ}  2^{k \alpha} 	\int_{S_k}   \frac{ 1}{|w + \bar z|^{4+ 2 \alpha}} 
	d\mu(w)
		\\
	\lesssim \sum_{n=- \infty} ^\infty  2^{n(2-q)}  2^{n(2+ \alpha)}  \quad \quad \quad  \quad \qquad  \quad \quad \quad  \quad \qquad  \quad \quad \quad  \quad \\
 \quad \quad \quad  \quad \qquad 
	\left( \sum_{k \le n}  2^{k\alpha}   2^{-n (3 + 2 \alpha)}      2^{k(q-1)}  \mathcal{C}_q [\mu_k] 
	        +   \sum_{k >n}      2^{k\alpha}    
	        	2^{-k (4 + 2 \alpha - q)}     \mathcal{C}_q [\mu_k]     \right) \\	 
	\lesssim \sum_{n=- \infty} ^\infty  2^{n(4 + \alpha -q)}  
	 \left( \sum_{k \le n}  2^{k(q+\alpha -1)}   2^{-n (3 + 2 \alpha)}    \mathcal{C}_q [\mu_k] 
	 +   \sum_{k >n}      	2^{-k (4 + \alpha - q)}    
	  \mathcal{C}_q [\mu_k]     \right) \\   
	 =
	 \sum_{k = - \infty}^\infty  \mathcal{C}_q [\mu_k]   
	 \left( 2^{k(q+\alpha -1)}  \sum_{n \ge k} 
	  2^{n(1 - \alpha -q)}    + 
	    2^{-k(4+\alpha -q)} \sum_{n <k} 
	  2^{n(4 + \alpha -q)} \right)  \\
	  \lesssim  \sum_{k = - \infty}^\infty  \mathcal{C}_q [\mu_k]  < \infty, \\
  \end{multline*}
	again, with implied constants depending on $\alpha$ and  $q$. Hence
	(\ref{eq:realBerezin}) holds. 
	
	Conversely, again
	 with $|I| = 2^{n+1}$, $z \in T_I$, we have
	 $$
	    B_\alpha \mu(z) \gtrsim 2^{n(2 +2 \alpha)}
	     \int_{T_I}  \frac{ 1}{|w + \bar z|^{4+ 2 \alpha}}  d \mu(w) \gtrsim 2^{-2n} \mu(T_I)
	 $$
	 and consequently
	 \begin{multline*}
	 	\int_0^\infty t^{1-q} \sup_{\re z =t} B_\alpha \mu(z) dt 
	 	\gtrsim \sum_{n = - \infty}^\infty 2^{n(2 - q)} 2^{-2n} 2^{nq}  \mathcal{C}_q[\mu_n] 
	 	=  \sum_{n = - \infty}^\infty   \mathcal{C}_q[\mu_n].
	 	\end{multline*}
	 	\color{black}
This finishes the proof of Theorem \ref{thm:CharacterizationLpToLq}.
\end{proof}

In general, applying Theorem~\ref{thm:SufficiencyLPhitoLq} with $\Phi(t)=t^{p}$, and computing the norms $\|\exp^{-q'2^{n-1}}\|_{\mathrm{L}^{\tilde \Phi^c}}=\|\exp^{-q'2^{n-1}}\|_{\mathrm{L}^{(p/q')'}}$, we obtain the following result, which we state for the sake of being explicit.
\begin{proposition}\label{thm:SufficiencyLpToLq}
	Let $q\ge 2$ and $p\ge q'$. With $\mu_n$ as in Theorem~\ref{thm:CharacterizationLpToLq}, it then holds that
	\begin{equation}\label{eq:SufficientLpToLq}
		\|\mathcal{L}\colon \mathrm{L}^{p}(0,\infty)\to \mathrm{L}^{q}(\CC_+,d\mu)\|^q\lesssim
		\sum_n 2^{nq/p}\mathcal{C}_{q}[\mu_n].
	\end{equation}
\end{proposition}
For $p<\infty$, condition \eqref{eq:SufficientLpToLq} is not necessary for $\mathcal{L}\colon \mathrm{L}^{p}(0,\infty)\to \mathrm{L}^{q}(\CC_+,d\mu)$ to be bounded, as can be seen from \cite[Theorem~3.5]{jpp13}. Next we will show that if $\mu$ has support on a vertical strip, then \eqref{eq:SufficientLpToLq} reduces to \eqref{eq:NecessaryIntensity}. This is the content of Theorem \ref{thm:pqexp1}, which is basically a reformulation of \cite[Thm.~3.6]{jpp13}, but allowing specifically for the case $p=\infty$.

\begin{theorem}\label{thm:pqexp1}
Let $\mu$ be a positive regular Borel measure supported in a strip
 $\CC_{\alpha_1, \alpha_2} = \{ z \in \CC: \alpha_1 \le \re z \le  \alpha_2 \}$ for some $\alpha_2 \ge \alpha_1  >0$, and let $1 \le p' \le q < \infty$ and $q \ge 2$.  
 Then the following assertions are equivalent:
\begin{enumerate}[(i)]
\item The embedding
$    \LL: \mathrm{L}^{p}(0,\infty) \rightarrow \mathrm{L}^{q}(\CC_+, \mu)$
is well-defined and bounded.
\item There exists a constant $C>0$ such that 
\begin{equation}  
 \label{eq:duren16}
\mu(Q_I) \le C
  |I|^{q/p'}  \text{ for all intervals }   I \subset i\RR.
  \end{equation}
  \end{enumerate}
In this case, the bound in $(i)$ depends only on $C$ and $\alpha_2/\alpha_1$.
\end{theorem}

\proof
Condition $(ii)$ is a reformulation of $\mathcal{C}_{q/p'}[\mu]<\infty$. The implication $(i)\implies (ii)$ was proved already in relation to \eqref{eq:NecessaryIntensity}. To obtain the reverse implication, assume instead that $\mathcal{C}_{q/p'}[\mu]<\infty$. Since $\mu$ is supported on a vertical strip, $\mu=\sum_{n=M}^N\mu_n$ for some integers $M,N$, with $N-M$ only depending on $\alpha_2/\alpha_1$. Hence,
\[
\sum_n 2^{nq/p}\mathcal{C}_{q}[\mu_n]=\sum_{n=M}^N 2^{nq/p}\mathcal{C}_{q}[\mu_n].
\]
By Lemma~\ref{lemma:GeometricPropertiesOfIntesities}, $2^{nq/p}\mathcal{C}_{q}[\mu_n]\approx \mathcal{C}_{q/p'}[\mu_n]$. Moreover, it is clear that $\mathcal{C}_{q/p'}[\mu_n]\le \mathcal{C}_{q/p'}[\mu]$. Thus, the above sum is finite, and $(i)$ follows from Proposition~\ref{thm:SufficiencyLpToLq}. \qed

\subsection{Laplace--Carleson embeddings on Orlicz spaces}

In addition to Theorem~\ref{thm:CharacterizationLpToLq}, we derive the following consequence of Theorem~\ref{thm:NecessaryLpToLq} and Theorem~\ref{thm:SufficiencyLPhitoLq}.
\begin{theorem}\label{thm:InfinityImpliesPhi}
	Assume that $q\ge 2$, and that $\mathcal{L}\colon \mathrm{L}^{\infty}(0,\infty)\to \mathrm{L}^{q}(\CC_{+},d\mu)$ is bounded. Then there exists an N-function $\Phi\colon[0,\infty)\to[0,\infty)$ for which $\mathcal{L}\colon \mathrm{L}^{\Phi}(0,\infty)\to \mathrm{L}^{q}(\CC_+,d\mu)$ is bounded.
\end{theorem}

We need some further lemmata to prove this result.

\begin{lemma}\label{lemma:ExponentialOrliczIntegral}
	Let $\Phi\colon[0,\infty)\to[0,\infty)$ be a Young function with left\--conti\-nuous derivative $\phi$. For $\alpha,C>0$ it then holds that
	\[
	\int_0^\infty \Phi\left(\frac{e^{-\alpha t}}{C}\right)\, dt=\frac{1}{\alpha C}\int_0^1\phi\left(\frac{s}{C}\right)\log\left(\frac{1}{s}\right)\, ds.
	\]
\end{lemma}
\begin{proof} Changing the order of integration,
\begin{eqnarray*}
	\int_0^\infty \Phi\left(\frac{e^{-\alpha t}}{C}\right)\, dt
	&=&
	\int_{t=0}^\infty \int_{s=0}^{\frac{e^{-\alpha t}}{C}}\phi(s)\, ds dt	\\
	&=&
	\int_{s=0}^{1/C} \phi(s)\int_{t=0}^{\frac{1}{\alpha}\log\left(\frac{1}{Cs}\right)}\, dtds\\
	&=&
	\frac{1}{\alpha}\int_{s=0}^{1/C} \phi(s)\log\left(\frac{1}{Cs}\right)\, ds.
\end{eqnarray*}
All that remains is the change of variables $Cs=s'$.
\end{proof}

\begin{lemma}\label{lemma:YoungFunctionExists}
	Let $q'\ge1$ and $(\gamma_n)_{n\in\ZZ}$ be a positive sequence such that $\gamma_n\ge 1$ for all $n\in \ZZ$, and $\gamma_n\to\infty$ as $|n|\to\infty$. Then there exists a N-function $\tilde \Phi^c$ such that 
	\[
	2^{n}\|\exp ^{-q'2^{n-1}}\|_{\mathrm{L}^{\tilde \Phi^c}}^{}\le \gamma_n\qquad (n\in\ZZ).
	\]
\end{lemma}

\begin{proof} Let $\phi^c\colon[0,\infty)\to[0,\infty)$ be a continuous, strictly increasing function with $\phi^c(0)=0$ and
\[
\phi^c(2^n)\le \frac{q'}{2}\frac{\gamma_n}{\int_0^1\log\left(\frac{1}{s}\right)\, ds}
\]
for all $n$. Such a function exists, since $\gamma_n\to\infty$ as $|n|\to\infty$. Define the Young function $\tilde \Phi^c\colon t\mapsto \int_0^t\phi^c(s)\, ds$, this is in fact an N-function as defined in (\ref{def:nfunction}). 
Using that each $\gamma_n\ge 1$, together with monotonicity, we find that
\[
\int_{s=0}^1\phi^c\left(\frac{2^{n}s}{\gamma_n}\right)\log\left(\frac{1}{s}\right)\, ds\le \phi^c\left(2^{n}\right)\int_{s=0}^1\log\left(\frac{1}{s}\right)\, ds \le\frac{q'}{2}\gamma_n.
\]
By Lemma~\ref{lemma:ExponentialOrliczIntegral}, the above left-hand side is equal to
\[
\frac{q'}{2}\gamma_n\int_0^\infty \tilde \Phi^c\left(\frac{2^{n}e^{-q'2^{n-1}t}}{\gamma_n}\right)\, dt,
\]
i.e.\ $2^{n}\|\exp ^{-q'2^{n-1}}\|_{\mathrm{L}^{\tilde \Phi^c}}^{}\le \gamma_n$ by the definition of the Orlicz norm.
\end{proof}

\begin{proof}[Proof of Theorem~\ref{thm:InfinityImpliesPhi}]
Since $\mathcal{L}\colon \mathrm{L}^{\infty}(0,\infty)\to \mathrm{L}^{q}(\CC_{+},d\mu)$ is bounded, it holds that $\sum_n\mathcal{C}_q[\mu_n]<\infty$ by Theorem~\ref{thm:NecessaryLpToLq}. Then there exists a positive sequence $(\gamma_n)_n$ such that $\gamma_n\to\infty$ and $\sum_n\gamma_n^{q-1}\mathcal{C}_q[\mu_n]<\infty$. We can assume without loss of generality that $\gamma_n\ge 1$ for every $n$. Let $\tilde\Phi^c$ be an $N$-function as in Lemma~\ref{lemma:YoungFunctionExists}, i.e.
\[
2^{n}\|\exp ^{-q'2^{n-1}}\|_{\mathrm{L}^{\tilde\Phi^c}}^{}\le \gamma_n.
\]
If $\Phi(t)=\tilde\Phi(t^{q'})$, then $\Phi$ is an $N$-function by Remark \ref{stabilityYoung}, and Theorem~\ref{thm:SufficiencyLPhitoLq} implies that $\mathcal{L}\colon \mathrm{L}^{\Phi}(0,\infty)\to \mathrm{L}^{q}(\CC_+,d\mu)$ is bounded.
\end{proof}

\subsection{Laplace--Carleson embeddings from $\mathrm{L}^{\Phi}(0,\tau_0)$}

In this section we develop finite time analogues of the preceding results on Laplace--Carleson embeddings. More precisely, we consider Laplace transforms of functions supported on $(0,\tau_0)$ for some $\tau_0>0$. We begin with the case of $\mathrm{L}^{\infty}(0,\tau_0)$, and then progress to $\mathrm{L}^{\Phi}(0,\tau_0)$ for more general Young functions $\Phi$. We will find that the value of $\tau_0$ is immaterial.

\begin{theorem}     \label{thm:FiniteT} 
Let $q\ge 2$, and $\mu$ be a positive regular Borel measure supported on $\CC_+$. Suppose that $\tau_0\in[2^M,2^{M+1}]$ for some integer $M$, and let $\mu^M$ denote the restriction of $\mu$ to the strip $\{0\le \re z\le 2^{-M}\}$. Then $\LL:\mathrm{L}^{\infty}(0,\tau_0) \to \mathrm{L}^{q}(\CC_+,\mu)$ is bounded if and only if
\begin{equation}    \label{eq:finiteC}
\sum_{n=-M}^\infty \mathcal{C}_q[\mu_n] +   \mathcal{C}_q[\mu^M] < \infty
\end{equation}
with an associated equivalence of norms, where the equivalence constant depends only on $q$. Moreover, if $\LL:\mathrm{L}^{\infty}(0,\tau_0) \to \mathrm{L}^{q}(\CC_+,\mu)$ is bounded, then $\LL:\mathrm{L}^{\infty}(0,\tau) \to \mathrm{L}^{q}(\CC_+,\mu)$ is bounded whenever $\tau>0$.
\end{theorem}

\proof We start by noting that it is sufficient to consider $\tau_0=2^M$. Indeed, if $\LL:\mathrm{L}^{\infty}(0,\tau_0) \to \mathrm{L}^{q}(\CC_+,\mu)$ is bounded, then $\LL:\mathrm{L}^{\infty}(0,2^M) \to \mathrm{L}^{q}(\CC_+,\mu)$ is bounded. The core of the proof is to prove that this is equivalent to \eqref{eq:finiteC}. It is easy to see that if we replace $M$ by $M+1$ in \eqref{eq:finiteC}, then we obtain an equivalent condition. This in turn implies boundedness of $\LL:\mathrm{L}^{\infty}(0,2^{M+1}) \to \mathrm{L}^{q}(\CC_+,\mu)$, and hence of $\LL:\mathrm{L}^{\infty}(0,\tau_0) \to \mathrm{L}^{q}(\CC_+,\mu)$. This argument immediately implies that boundedness of $\LL:\mathrm{L}^{\infty}(0,\tau_0) \to \mathrm{L}^{q}(\CC_+,\mu)$ yields boundedness of $\LL:\mathrm{L}^{\infty}(0,\tau) \to \mathrm{L}^{q}(\CC_+,\mu)$ for all $\tau>0$.

The proof that \eqref{eq:finiteC} is necessary is largely analogous to the proof of Theorem~\ref{thm:NecessaryLpToLq}.	

We fix $M \in \ZZ$. For $n \ge -M$, we define $f_n, F_n, N$ as in the proof of Theorem~\ref{thm:NecessaryLpToLq}. For $k=0, \dots, N-1$, define
$$
     g_k = \sum_{m \in \ZZ, m N +k \ge -M}f_{m N +k}     \quad \text{ and } G_k = \LL g_k.
$$
Note that $g_k \in \mathrm{L}^{\infty}(0, 2^M)$ for $k=0, \dots, N-1$.
As in the proof of Theorem~\ref{thm:NecessaryLpToLq}, we obtain:

If $n\in\ZZ$, $nN+k \ge {-M}$, and $z\in T_{nN+k}$, then $|G_k(z)|\ge \frac{1}{2}|F_{nN+k}(z)|\ge c2^{-nN-k-1}$. This implies, 
\begin{equation*}
		\sum_{n \ge -M}\mathcal{C}_q[\mu_n]
		\lesssim 
		\sum_{k=1}^N\sum_{\substack{n\in\ZZ;\\ nN+k \ge -M}}\int_{T_{nN+k}}|G_k|^q\, d \mu
		\le 
		\sum_{k=1}^N\int_{\CC_+}|G_k|^q\, d \mu.
	\end{equation*}
	Assuming that $\LL\colon \mathrm{L}^{\infty}(0,2^M)\to \mathrm{L}^{q}(\CC_+,\mu)$ is bounded, the above left-hand side is finite. 
	
	We still have to check boundedness of 
	the second term in Condition \eqref{eq:finiteC}. For any interval $I$ with $|I| = 2^{-M-1}$, let $c$ be the center and define $f=  \chi_{[0, 2^M]}(t) e^{ict}$, $F = \LL f$. Then
	$$
	   F(s) = \int_{0}^{2^M}   e^{-(s - ic)t} dt = \int_{0}^{2^M}   e^{-(\Re s) t}     e^{i(c-\Im s )t} dt.
	$$
	Note that $t \Re s \le \frac{1}{2}$ and $|c - \Im s| t \le \frac{1}{4}$ for $t \in [0, 2^M]$, $s \in Q_I$, thus
	$$
	    |F(s)| \gtrsim 2^{M+1}   \text{ for } s \in Q_I,
	$$
	and, using boundedness of $\LL:\mathrm{L}^{\infty}(0,2^M) \to \mathrm{L}^{q}(\CC_+,\mu)$,
\begin{equation*}
	      1 \gtrsim     \int_{\CC_+}|F(z)|^q\, d \mu(z)       \ge  \int_{Q_I}|F(s)|^q\, d \mu(s)  \\ \gtrsim     \frac{\mu(Q_I)}{|I|^q}.
\end{equation*}

We now turn to sufficiency of  (\ref{eq:finiteC}). Boundedness of the embedding for the measure $\sum_{n=-M}^\infty \mu_n$ follows directly from Theorem \ref{thm:CharacterizationLpToLq}.
To finish the proof,
it is sufficient to show that if $ \mu $ is supported on $[0,2^{-M}) \times \RR$ and $\mu(Q_I) \lesssim|I|^q$ for all $|I|=2^{-M+1}$, then
$\LL: \mathrm{L}^{\infty}(0,2^M) \to \mathrm{L}^{q}(\CC_+,\mu)$ is bounded.

Note that $\LL: \mathrm{L}^{\infty}(0,2^M) \to \mathrm{H}^2(\CC_{-2^{-M},+})$ is bounded with norm proportional to  $2^{M/2}$,
since the function $t \mapsto e^{-t \Re s} f(t)$ lies in $\mathrm{L}^2(0,2^M)$ when $f \in \mathrm{L}^{\infty}(0,2^M)$, with the corresponding norm estimate.
Here $\mathrm{H}^2(\CC_{-2^{-M},+})$ is the Hardy space on the larger half-plane $\{s: \Re s > -2^{-M} \}$.

We observe that for the norm of the embedding $\EE$, we have
$$
      \| \EE \|_{\mathrm{H}^2(\CC_{-2^{-M},+}) \to \mathrm{L}^{q}(\mu)} =  \| \EE \|_{\mathrm{H}^2(\CC_{+}) \to \mathrm{L}^{q}(\tilde \mu_{2^{-(M+1)}})},
$$
where
$$
   \tilde \mu_{2^{-(M+1)}}(E) = \mu(E- 2^{-(M+1)}).
$$
Now $\tilde \mu_{2^{-(M+1)}}$ is supported on the strip $S_{-M-1}$, and we may directly apply Theorem~\ref{thm:pqexp1} in order to obtain that
 $\| \EE \|_{\mathrm{H}^2(\CC_{-2^{-M},+}) \to \mathrm{L}^{q}(\mu)} \lesssim 2^{-M/2}$. This finishes the proof.
\qed

\begin{theorem}\label{thm:InfinityImpliesPhiFiniteTime}
	Assume $q\ge 2$, $\tau_0>0$, and that $\mathcal{L}\colon \mathrm{L}^{\infty}(0,\tau_0)\to \mathrm{L}^{q}(\CC_{+},d\mu)$ is bounded. Then there exists an N-function $\Phi\colon[0,\infty)\to[0,\infty)$ for which $\mathcal{L}\colon \mathrm{L}^{\Phi}(0,\tau_0)\to \mathrm{L}^{q}(\CC_+,d\mu)$ is bounded.
\end{theorem}
\proof

By Theorem~\ref{thm:FiniteT}, we may assume that $\tau_0=2^M$ for some integer $M$. Let $\mu=\mu'+\mu^M$, where $\mu'=\sum_{n=-M}^\infty\mu_n$. Assuming boundedness of $\mathcal{L}\colon \mathrm{L}^{\infty}(0,2^M)\to \mathrm{L}^{q}(\CC_{+},d\mu)$, condition \eqref{eq:finiteC} together with Theorem~\ref{thm:CharacterizationLpToLq} implies boundedness of 
\[
\mathcal{L}\colon \mathrm{L}^{\infty}(0,\infty)\to \mathrm{L}^{q}(\CC_{+},d\mu').
\]
By Theorem~\ref{thm:InfinityImpliesPhi}, there exists an N-function $\Phi$ such that $\mathcal{L}\colon \mathrm{L}^{\Phi}(0,\infty)\to \mathrm{L}^{q}(\CC_{+},\mu')$, and since $\mathrm{L}^{\Phi}(0,2^M)\hookrightarrow \mathrm{L}^{\Phi}(0,\infty)$ isometrically, $\mathcal{L}\colon \mathrm{L}^{\Phi}(0,2^M)\to \mathrm{L}^{q}(\CC_{+},\mu')$ is bounded.

The proof will be complete once we have established that $\mathcal{L}\colon \mathrm{L}^{\Phi}(0,2^M)\to \mathrm{L}^{q}(\CC_{+},\mu^M)$ is bounded. Note that the N-function $\Phi$ obtained in the proof of Theorem~\ref{thm:InfinityImpliesPhi} is of the form $\Phi(t)=\tilde\Phi(t^{q'})$ for some other N-function $\tilde\Phi$. By H\"older's inequality for Orlicz spaces, it follows that $\mathrm{L}^{\Phi}(0,2^M)\hookrightarrow \mathrm{L}^{q'}(0,2^M)$. Repeating an argument from the proof of Theorem~\ref{thm:FiniteT}, $\LL\colon \mathrm{L}^{\Phi}(0,2^M)\to \mathrm{H}^q(\CC_{-2^{-M}})$, and $\mathrm{H}^q(\CC_{-2^{-M}})\hookrightarrow \mathrm{L}^{q}(\CC_+,d\mu^M)$, again by condition \eqref{eq:finiteC}.
\qed

\begin{corollary}    \label{thm:ZAdm}
Let $q\ge 2$, and $\mu$ be a positive regular Borel measure supported on $\CC_+$. 
Suppose that $\LL:\mathrm{L}^{\infty}(0,\tau_0) \to \mathrm{L}^{q}(\CC_+,\mu)$ is bounded for some $\tau_0 >0$.
Then
$$
   \lim_{\tau \to 0} \| \LL \|_{ \mathrm{L}^{\infty}(0, \tau) \to \mathrm{L}^{q}(\CC_+, \mu)} = 0.
$$
In fact, with $\Phi$ as in Theorem~\ref{thm:InfinityImpliesPhiFiniteTime}, it holds that 
\[
	\| \LL \|_{ \mathrm{L}^{\infty}(0, \tau) \to \mathrm{L}^{q}(\CC_+, \mu)}\le \| \LL \|_{ \mathrm{L}^{\Phi}(0, \tau_0) \to \mathrm{L}^{q}(\CC_+, \mu)}\|\chi_{[0,\tau]}\|_{\mathrm{L}^{\Phi}(0,\infty)}
\]
whenever $\tau\in (0,\tau_0]$.
\end{corollary}

\proof
Let $f\in \mathrm{L}^{\infty}(0,\infty)$ have unit norm, and support on $(0,\tau)$. With $\Phi$ as in Theorem~\ref{thm:InfinityImpliesPhiFiniteTime},
\[
	\|\LL f\|_{\mathrm{L}^{q}(\CC_+, \mu)}\le \| \LL \|_{ \mathrm{L}^{\Phi}(0, \tau_0) \to \mathrm{L}^{q}(\CC_+, \mu)}\|f\|_{\mathrm{L}^{\Phi}(0,\tau_0)}.
\]
The desired estimate now follows from $\|f\|_{\mathrm{L}^{\Phi}(0,\tau_0)}\le\|\chi_{[0,\tau]}\|_{\mathrm{L}^{\Phi}(0,\infty)}$.
\qed

\subsection{A Laplace--Carleson embedding for a specific class of Orlicz spaces}
In this section, we want to present applications of the theory developed above to some concrete Orlicz spaces.

In the following let 
\[\Phi(t)=\Phi_{\exp}(t):=\exp t-t-1\qquad\text{and}\qquad \tilde\Phi(t)=\tilde\Phi_{\exp}(t):=\exp \sqrt{t}-\sqrt{t}-1,\] so that $\Phi(t)=\tilde\Phi(t^2)$.
We will show that for this specific Young function, the boundedness of the Laplace--Carleson embedding from $\mathrm{L}^{\Phi}(0,1)$ to $\mathrm{L}^{2}(\CC_+,\mu)$ can be characterized in terms of the capacity, in an analogous way as in Theorem \ref{thm:CharacterizationLpToLq} for $\mathrm{L}^{\infty}$. 

\begin{theorem}    \label{thm:finiteTexp}
Let $\mu$ be a positive regular Borel measure  on $\CC_+$.
Then $\LL:\mathrm{L}^{\Phi}(0,1) \to \mathrm{L}^2(\CC_+,\mu)$ is bounded, if and only if

\begin{equation}    \label{eq:finiteC1}
	\sum_{n=1}^\infty n^2\mathcal{C}_2[\mu_n]	+ \sup_{I \text{ interval, } |I| =2 } \mu(Q_I)<\infty
\end{equation}
with an associated equivalence of norms.
\end{theorem}

\proof
To prove the necessity we will reuse some notation and quantities from the proof of Theorem~\ref{thm:NecessaryLpToLq}. In particular, for each integer $n\ge 2$ we let $T_n$ denote the right half of a Carleson square $Q_{I_{n}}$ with side length $2^{n+1}$ and $\mathcal{C}_2[\mu_n]\le 2^{3-2n}\mu_n(T_n)$. Moreover, the functions $f_m=\chi_{(2^{-m-1},2^{-m}]}(t)e^{ic_mt}$, with $ic_m$ being the midpoint of $I_{{m}}$,  are $\mathrm{L}^{\infty}$-normalized functions with disjoint supports such that $F_m=\LL f_m$ is essentially localized to $T_m$: There exists $c,C>0$ for which
\begin{equation}   \label{eq:localization}
	z\in T_m\implies |F_m(z)|\ge c2^{-m}\quad \textnormal{and}\quad |F_n(z)|\le C2^{-m-|m-n|}.
\end{equation}

For a given $\epsilon>0$, we may choose $N$ such that
\[
	\sum_{m=1}^{n-1}m2^{mN}\le \epsilon n2^{nN}\quad \textnormal{and}\quad \sum_{m=n+1}^{\infty}m2^{-mN}\le \epsilon n2^{-nN}
\]
uniformly in $n$. This can be seen by comparison with a Riemann integral. For such an $N$ and $k \in \{0, \dots, N-1\}$, let 
\begin{equation}   \label{eq:gdef}
g_k = ( \log 2 )\sum_{m=0}^\infty   m  f_{k+ mN}.
\end{equation}
and write $G_k= \LL g_k$. Note that
\[
\int_0^1 \Phi(|g_k(t)|) dt \leq \int_0^1 e^{|g_k(t)|} dt = \sum_{m=0}^\infty  2^{-(k+mN+1)} 2^m \le 1,
\]
whence $\|g_k\|_{\Phi} \le 1$. Moreover, for $z \in T_{k+nN}$
\begin{eqnarray*}
	\sum_{m \in \ZZ, m \ge 0, n \neq m}  m  |F_{k+ mN}(z)| &\le& C  \sum_{m \in \ZZ, m \ge 0, n \neq m} m 2^{- (k+nN + |m-n|N)} \\
	&=& C  \sum_{m = 0}^{n-1} m 2^{- (k+2nN -mN)} \\&& +\  C  \sum_{m = n+1}^{\infty}  m 2^{- (k+mN)}\\
	&\le& C \epsilon n 2^{1-k-nN}\\
	&\le& \frac{2C\epsilon n}{c}|F_{k+nN}(z)|,
\end{eqnarray*}
and hence
\begin{eqnarray*}
	|G_k(z)| &\ge&  (\log 2) \left( n  |F_{k+ nN}(z)| - \sum_{m \in \ZZ, m \ge 0, n \neq m}  m  |F_{k+ mN}(z)|  \right) \\
	&\gtrsim&  n |F_{k+nN}(z)| \ge cn 2^{-k-nN},
\end{eqnarray*}
provided that $\epsilon$ is sufficiently small. A possible choice is $\epsilon = \frac{1}{8} \frac{c}{C}$, where $c,C$ are the constants from (\ref{eq:localization}).

Hence for $k \in \{0, \dots, N-1\}$,
$$
\sum_{n=1}^\infty   n^2 \mathcal{C}_2[\mu_{k+ nN}] \lesssim   \sum_{n=1}^\infty  n^2 2^{-2(k+nN)} \mu(T_{nN+k}) \lesssim_N \int_{\CC_+} |G_k(z)|^2 d\mu.
$$
Adding over $k=0, \dots, N-1$, we obtain the required norm bound of the first term in (\ref{eq:finiteC1}), with a constant only depending on $N$ (therefore on $\epsilon$, and hence only on $c, C$).  To control the second term in (\ref{eq:finiteC1}), just consider $f= e^{it c_I} \chi_{(0,1)}$,
where $c_I$ is the midpoint of the interval $I$.
\medskip

To prove the sufficiency of Condition \eqref{eq:finiteC1}, note first that boundedness of
$$
\LL:\mathrm{L}^{\Phi}(0,1) \to \mathrm{L}^2(\CC_+,\mu^{-1})
$$
follows immediate from the continuous embedding $\mathrm{L}^{\Phi}(0,1) \subset L^2(0,1)$, together with
the Carleson Embedding Theorem for Paley--Wiener spaces, see e.g. \cite{PPZ}. Here, as in the notation of Theorem \ref{thm:FiniteT}, $\mu^{-1}$ denotes the restriction of $\mu$ to the strip $\{z \in \CC:  0 \le \Re z \le 2 \}$.

For the remaining part of the measure $\mu$, one may use
a straightforward adaptation of Theorem~\ref{thm:SufficiencyLPhitoLq},
\begin{equation*}
	\|\mathcal{L}\colon \mathrm{L}^{\Phi}(0,1)\to \mathrm{L}^2(\CC_{+},d\mu)\|^2\lesssim
	\sum_{n=1} 2^{n}\|\exp^{-2^{n}}\|_{\mathrm{L}^{\tilde \Phi^c}(0,1)}\mathcal{C}_{2}[\mu_n].
\end{equation*}
Thus, in order to conclude the proof of sufficiency, it is enough to establish the estimate
\[
	2^n\|\exp^{-2^n}\|_{\mathrm{L}^{\tilde{\Phi}^c}(0,1)}\lesssim n^2 \qquad \forall n\in \NN.
\]
We thus need to show that for sufficiently large $B$, it holds that
\[
\int_0^1\tilde{\Phi}^c\left(\frac{2^n\exp(-2^nt)}{Bn^2}\right)\, dt\le 1.
\]

This is indeed possible but requires a somewhat arduous explicit computation.
It suffices to do this for large $n$, since the above integral is always finite and converges to $0$ as $B\to\infty$.
It is straightforward to see that \[\tilde\phi(t):=\tilde{\Phi}'(t)=\frac{\exp\sqrt{t}-1}{2\sqrt{t}},\] and by a comparison of power series,
\[
\frac{\exp\left(\sqrt{t}/2\right)}{2}\le \tilde{\phi}(t)\le\frac{\exp\left(\sqrt{t}\right)}{2}.
\]
It follows that $\tilde\phi^c$, the left-continuous inverse of $\tilde\phi$, vanishes on $[0,1/2]$, and satisfies 
\[
\left(\log\left(2t\right)\right)^2\le \tilde \phi^c(t)\le 4\left(\log\left(2t\right)\right)^2
\]
for $t>1/2$. If one defines
\[
\Psi(t)=\begin{cases}
	0,& t\in[0,1/2],\\
	4\int_{1/2}^t\left(\log\left(2s\right)\right)^2\, ds, & t>1/2,
\end{cases}
\]
then $\tilde{\Phi}^c(t)\le \Psi(t)$. Assuming $n$ is sufficiently large for $2^n\exp(-2^n)/Bn^2<1/2$, we apply Fubini's theorem to obtain
\begin{eqnarray*}
	\int_0^1\tilde{\Phi}^c\left(\frac{2^n\exp(-2^nt)}{Bn^2}\right)\, dt &\le &
	\int_0^{\frac{1}{2^n}\log\left(\frac{2^{n+1}}{Bn^2}\right)}\Psi\left(\frac{2^n\exp(-2^nt)}{Bn^2}\right)\, dt
	\\
	&=&
	4\int_{t=0}^{\frac{1}{2^n}\log\left(\frac{2^{n+1}}{Bn^2}\right)}\int_{s=1/2}^{\frac{2^n\exp(-2^nt)}{Bn^2}}\left(\log\left(2s\right)\right)^2\, dsdt
	\\
	&=&
	4\int_{s=1/2}^{2^n/Bn^2}\left(\log\left(2s\right)\right)^2\int_{t=0}^{\frac{1}{2^n}\log\left(\frac{2^{n+1}}{Bn^2s}\right)}\, dtds
	\\
	&=&
	4\int_{s=1/2}^{2^n/Bn^2}\left(\log\left(2s\right)\right)^2\frac{1}{2^n}\log\left(\frac{2^{n+1}}{Bn^2s}\right)\, ds.
\end{eqnarray*}
Through a rather arduous calculation, one finds the limit of the above integral as $n\to\infty$ to be $4\left(\log 2\right)^2/B$. 
\qed

 As an alternative to the concrete calculations in the proof above, we may take a slightly different path and observe that
\begin{lemma} Let $N \in \ZZ$, $N \ge 0$ and let $\Phi(t)=\exp t-t-1$ for $t \ge 0$. Then 
	$\LL:  \mathrm{L}^{\Phi}(0,1) \rightarrow H^2(\CC_{+,{2^N}}) $
	is bounded with norm 
	$$
	\| \LL \|_{  \mathrm{L}^{\Phi}(0,1)\ \to H^2(\CC_{+,{2^N}})} \lesssim N \frac{1}{2^{N/2}}.
	$$
\end{lemma}
\begin{proof}
	Let $\| f \|_{ \mathrm{L}^{\Phi}} =1$. Note that by the Paley--Wiener Theorem, it is enough to prove that
	\begin{equation}  \label{eq:Nest}
	\|   f  \exp^{- 2^N }  \|_2 \lesssim N  \frac{1}{2^{N/2}}.
		\end{equation}
	Let $p, q >2$ with $\frac{1}{p} + \frac{1}{q} = \frac{1}{2}$. Then by H\"older's inequality,
	$$
	\|   f  \exp^{- 2^N  }  \|_2  \lesssim   \|f \|_{p}    \|   \exp^{- 2^N  }\|_{q} \lesssim p  \frac{1}{2^{N/q}},
	$$
	with constants independent of $p$, $N$, where we use the weak exponential integrability of $f$,
	$$
	|\{ t \in (0,1): |f(t)| > \alpha \}| \lesssim  e^{-\alpha},
	$$
	and standard estimates of the $\Gamma$-function. Choosing $p=N$,
we find in case $N \ge 2$ 
	$$
	\|   f  e^{- 2^N t }  \|_2 \lesssim   N \frac{1}{2^{N/2}}.
	$$
	In case $N=0,1$, the estimate follows trivially from $  \mathrm{L}^{\Phi}(0,1) \subset L^2(0,1)$.
\end{proof}
	
To finish the proof of Theorem \ref{thm:finiteTexp}, note that the embedding
$$
H^2_{\CC_{+,2^n}} \rightarrow L^2(\mu_{n+1})
$$
has norm equivalent to  $\left(C_2[\mu_{n+1}]\right)^{1/2}$ by the classical Carleson Embedding Theorem, applied to the shifted half-plane $\CC_{+,2^n}$.

The remainder follows now from the decomposition of $\CC_+$ into the strips $S_n$, $n \ge 1$, together with the strip $ \{ z \in \CC_+: 0 \le \Re z \le 2\}$,
and the inclusion $  \mathrm{L}^{\Phi} \subset L^2(0,1)$:

\begin{eqnarray*}
\| \LL f \|^2_{\mathrm{L}^2(\CC_+,\mu)} &\le& 
\| \LL f \|^2_{\mathrm{L}^2(S,\mu)}  + \sum_{N \ge -1} \| \LL f \|^2_{\mathrm{L}^2(\CC_+,\mu_{N+1})}\\
&\lesssim& \|f \|_2^2 + \sum_{N \in \ZZ} 2^N C_2[\mu_{N+1}] \|\LL f \|^2_{H^2_{\CC_{+,2^N}}} \\
	&\lesssim& \left(1 +\sum_{N \ge 0} N^2 C_2[\mu_{N}] \right) \| f\|_{ \mathrm{L}^{\Phi}}.
\end{eqnarray*}
This proof extends without difficulty to the case of the Young function
$\Phi_\alpha(t) = \exp(t^\alpha) - t^\alpha -1$ on $[0,1]$, where $\alpha \ge 1$.

Using analogous estimates and choosing $p= N \alpha$ in the application of H\"older's inequality, we
obtain the correct analogue of
$ (\ref{eq:Nest})$:

\begin{equation}
		\|   f  \exp^{- 2^N }  \|_2 \lesssim N^{1/\alpha}  \frac{1}{2^{N/2}}   \text{ for }  \| f \|_{ \mathrm{L}^{\Phi_\alpha}} \le 1  .
\end{equation}

The rest of the sufficiency proof follows as above. The proof of necessity again follows along the same
lines, replacing the test function $g_k$ in  (\ref{eq:gdef}) by
\begin{equation}  
	g_k = ( \log 2 )^{1/\alpha} \sum_{m=0}^\infty   m^{1/\alpha}  f_{k+ mN}.
\end{equation}

Altogether, we obtain
\begin{theorem}    \label{thm:finiteTexpalpha}
	Let $\mu$ be a positive regular Borel measure supported on $\CC_+$ and let $\alpha >1$. 
	Then $\LL:\mathrm{L}^{\Phi_\alpha}(0,1) \to \mathrm{L}^2(\CC_+,\mu)$ is bounded, if and only if
		\begin{equation}    \label{eq:finiteexpalpha}
		\sum_{n=1}^\infty n^{2/\alpha} \mathcal{C}_2[\mu_n]   	+ \sup_{I \text{ interval, } |I| =2 } \mu(Q_I) <\infty	
	\end{equation}
	with an associated equivalence of norms.
\end{theorem}

\begin{remark}
	An inspection of the proof above reveals that the implied constants can be chosen independent of $\alpha$. Hence Theorem \ref{thm:CharacterizationLpToLq} may (in case $q=2$) be obtained as a limiting case of Theorem \ref{thm:finiteTexpalpha}, in the limit $\alpha \to \infty$.
\end{remark}

\section{Admissible operators}\label{sec2}

In this section we draw the connection of the derived results on Laplace--Carleson embeddings to linear control systems of the form (\ref{eq:ABsystem}),
\begin{equation}
	\dot x(t)=Ax(t)+Bu(t), \qquad x(0)=0, \quad t \ge 0.
\end{equation}

Here $(T(t))_{t\ge 0}$ is a $C_0$-semigroup of bounded linear operators on a Banach space $X$. Its infinitesimal generator is denoted by $A$, which is a closed operator with dense domain $D(A)$. As is well-known, the semigroup $(T(t))_{t \ge 0}$ has a unique extension to $X_{-1}$, which is the completion of $X$ with respect to the norm $\|(\beta-A)^{-1}\cdot\|_X$, where $\beta\in\rho(A)$ is fixed but arbitrary. With a slight abuse of notation, this extension is again denoted by $(T(t))_{t \ge 0}$. For future reference, we note that if $X$ is additionally reflexive, then that $X_{-1}$ is isomorphic to the dual of $D(A^*)$, provided that one uses the duality pairing of $X$ and $D(A^*)$ is equipped with the graph norm. Finally, $B\colon U\to X_{-1}$ is bounded and linear, where $U$ is a Banach space, and $Z(0,t_0;U)$ is a Banach space of $U$-valued functions on $(0,t_0)$. Our main interest is when $Z$ is an Orlicz type space $L^\Phi$, of which $L^p$ is a special case. 

By a closed graph argument and the semigroup property, the notion of an admissible operator $B$ can be rephrased as follows.
\begin{definition}
Let $Z$ be an Orlicz space.  An operator $B \in L(U,X_{-1})$ is called {\em $Z$-admissible} (for $(T(t))_{t \ge 0}$ or $A$),  %
if for all $t_{0}>0$ and all $u\in Z(0,t_{0};U)$ it holds that
  \[ \Theta u=\Theta_{t_{0}}u=\int_0^{t_{0}} T(t_{0}-s)B u(s) \, \mathrm{d}s\in X.\]
Furthermore, we define the following two refinements of admissibility.
We say that 
\begin{itemize}
\item $B$  is {\emph{zero-class} $Z$-admissible}, if $\lim_{t_{0}\to0^{+}}\|\Theta_{t_{0}}\|_{\mathcal{L}(Z(0,t_{0};U),X)}=0$, and 
\item $B$ is {\emph{infinite-time} $Z$-admissible}, if $\sup_{t_{0}>0}\|\Theta_{t_{0}}\|_{\mathcal{L}(Z(0,t_{0};U),X)}<\infty$.
\end{itemize}

 \end{definition}
 Note that $B$ is infinite-time $Z$-admissible if and only if the operator
 \[ Z(0,\infty;U)\to X, u\mapsto \int_{0}^{\infty}T(s)Bu(s)\mathrm{d}s\]
 is bounded. 
We further mention that admissibility may be studied for other choices of function spaces $Z$, such as weighted $\mathrm{L}^{p}$-spaces \cite{HLM} and Sobolev spaces \cite{jpp12}. The interest in Orlicz spaces arises in the connection of admissibility to (integral) input-to-state stability for infinite-dimensional systems, see \cite{JNPS}.

Unbounded admissible operators, that is, operators $B$ not bounded as a mapping from $U$ to $X$,  naturally appear in the study of boundary control of evolution equations. 
The most commonly studied case in the literature is $Z=\mathrm{L}^{2}$ and we refer to the
survey \cite{jp} and the book \cite{TW} for the basic background to admissibility in the context of well-posed and boundary control systems. 
The general case was already studied in the seminal works by Weiss \cite{Weiss89i,Weiss89ii}, where the notion of ``admissibility'' was coined, although it had appeared earlier, e.g.\ \cite{Salamon}. See also \cite{haakthesis}, where several results previously known for $p=2$ were generalized. Admissible operators with respect to Orlicz spaces, $Z=\mathrm{L}^{\Phi}$, were studied in \cite{JNPS} and we refer to that paper for elementary facts of $Z$-admissible operators.

It is easy to see that the property of admissibility does not depend on the choice of $t_{0}$, which justifies the fact that we omit the reference to $t_{0}$ in the operator $\Theta$.
Let us fix the following notation for a semigroup generator $A$ on $X$:
\[
\mathfrak{B}_{Z}(A,U)=\{B\in L(U,X_{-1})\colon B\text{ is an }Z\text{-admissible control operator for }A\}.
\]
 The inclusions 
	\begin{equation}\label{eq:inclusions}
   		\mathfrak{B}_{\mathrm{L}^{1}}(A,U)   \subseteq \mathfrak{B}_{\mathrm{L}^{p}}(A,U)\subseteq \mathfrak{B}_{\mathrm{L}^{\infty}}(A,U), \quad p\in[1,\infty],
	\end{equation} 
are clear by the nesting properties of Orlicz spaces.
A question in which we are particularly interested in is when $\mathfrak{B}_{\mathrm{L}^{p}}(A,U)=\mathfrak{B}_{\mathrm{L}^{\infty}}(A,U)$ for some $p\in[1,\infty)$. This is non-trivial as examples of semigroups are known for which all inclusions in \eqref{eq:inclusions} are strict (for all $p\in[1,\infty)$), see \cite[Example 5.2]{JNPS} and \cite{jpp12}. One should note that these are examples on Hilbert spaces $X$, whereas the following, simpler, example shows that the situation on Banach spaces only becomes worse.

\begin{example}\label{ex1}
	Let $X=L^p(0,\infty)$ with $1<p<\infty$, and $U=\mathbb{C}$. The right-shift semigroup $(T_{p}(t))_{t\ge0}$ on $X$, with generator $A_{p}$, is defined by the dual action $T_{p}(t)^*f(x)=f(t+x)$ for $t,x\ge0$. Similarly, define $B\colon\mathbb{C}\to D(A_p^*)'$ by its dual action $B^*f=f(0)$ for $f\in D(A_{p}^*)$. By calculation,
	\[
	\langle\Theta u,\varphi\rangle_{L^p(0,\infty)} 
	=
	\int_0^{t_0}\langle T_p(t_0-s)Bu(s),\varphi\rangle_{L^p}\,\mathrm{d}s
	=
	\int_0^{t_0}u(s)\varphi(t_0-s)\,\mathrm{d}s
	\]
	whenever $\varphi$ is smooth with compact support. It follows that $B$ is $\mathrm{L}^{p}$-admissible, but fails to be $\mathrm{L}^{q}$-admissible for any $q<p$. In particular, this shows that $\mathfrak{B}_{\mathrm{L}^{q}}(A_{2},\mathbb{C})\subsetneq\mathfrak{B}_{\mathrm{L}^{2}}(A_{2},\mathbb{C})$ for $q<2=p$.
\end{example}

\subsection{ Left-invertible semigroups on Hilbert spaces}

In \cite{Weiss89ii}, it was (implicitly) shown that $\mathfrak{B}_{\mathrm{L}^{\infty}}(A,\CC)=\mathfrak{B}_{\mathrm{L}^2}(A,\CC)$ for $A$ being the periodic left-shift semigroup on $\mathrm{L}^{2}(0,2\pi)$, corresponding to the control of a one-dimensional wave equation. It turns out that this result holds true in a much more general setting. This is a rather direct consequence of another result by G.~Weiss, which was derived in the context of what later became known as the Weiss conjecture, \cite{WeissConjectures91}.
\begin{theorem}\label{thm1}
Let $A$ generate a left-invertible semigroup on a Hilbert space $X$. Then for any Hilbert space $U$ it holds that
$$\mathfrak{B}_{\mathrm{L}^{\infty}}(A,U)=\mathfrak{B}_{\mathrm{L}^2}(A,U).$$
\end{theorem}
\begin{proof}
This basically follows by \cite[Theorem 4.1]{WeissConjectures91} which is a slight generalization of an older result by Hansen and Weiss \cite{HansenWeiss}. In fact, let $B\in\mathfrak{B}_{\mathrm{L}^{\infty}}(A,U)$. Then, it follows by the definition of $\mathrm{L}^{\infty}$-admissibility and the Laplace transform that 
$$\sup_{\Re\lambda>\alpha}\|(\lambda \mathrm{I}-A)^{-1}B\|<\infty$$
for some $\alpha\in\RR$. By (the dual version of) \cite[Theorem 4.1]{WeissConjectures91}, see also \cite{zwart2005}, this implies that $B$ is $\mathrm{L}^{2}$-admissible.
\end{proof}

Example \ref{ex1} shows that the assumption that $X$ is a Hilbert space in Theorem \ref{thm1} cannot be dropped in general, even in the specific case of $U=\CC$. However, for the specific case of $A_{{r},\mathrm{per}}$ being the generator of the periodic left-shift semigroup on $\mathrm{L}^{r}(0,2\pi)$, characterizations of $\mathfrak{B}_{\mathrm{L}^{p}}(A_{r,\mathrm{per}},\CC)$ can be derived from results on Fourier multipliers, \cite[Proposition 5.2]{Weiss89ii}; more precisely, 
\begin{eqnarray*}
\lefteqn{\mathfrak{B}_{\mathrm{L}^{p}}(A_{r,\mathrm{per}},\CC)}\\
&=&\{b\in S_{\mathrm{per}}[0,2\pi]\colon f\mapsto \sum_{k\in\ZZ}\hat{b}(k)\hat{f}(k)\mathrm{e}^{ikt} \in \mathcal{L}(\mathrm{L}^{p}(0,2\pi),\mathrm{L}^{r}(0,2\pi))\},
\end{eqnarray*}
where $\hat{h}(k)$ denotes the $k$-th Fourier coefficient and $S_{\mathrm{per}}[0,2\pi]$ the periodic distributions on $[0,2\pi]$. By known facts on multipliers, this implies in particular that 
\[\mathfrak{B}_{\mathrm{L}^{p}}(A_{r,\mathrm{per}},\CC)=\mathfrak{B}_{\mathrm{L}^{\infty}}(A_{r,\mathrm{per}},\CC)\qquad\text{ for all }p\ge2 \text{ and }r\leq2,\]
which generalizes the assertion of Theorem \ref{thm1} in the situation of this special generator. This observation motivates 
studying the relation of the sets $\mathfrak{B}_{\mathrm{L}^{p}}(A,\mathbb{C})$ for more general group generators of {\em diagonal} form. This is treated in the next section. Also note that the facts on Fourier multipliers used above give a glimpse on why the relation between the sets $\mathfrak{B}_{\mathrm{L}^{p}}(A,U)$ for different $p$ is non-trivial in general. 
We shall see a related result in
 Theorem \ref{mainsec3} below.

\subsection{Diagonal $C_0$-semigroups}
\label{sec:diagroups}

In this section we assume that the semigroup generator $A$ is diagonal with respect to a (Schauder) basis of $X$. More precisely, fix
 $1 \le q < \infty$ and a $q$-Riesz basis $(\phi_k)_{k\in\ZZ}$ of $X$, i.e.,
 for some $C_1, C_2 > 0$ we have that for all finite sequences $(a_{k})_{k}$,

\[
C_1 \sum_{k} |a_k|^q \le \| \sum_{k} a_k \phi_k \|^q \le C_2 \sum_{k} |a_k|^q.
\]
Let $A:D(A)\subset X\to X$ be an operator such that $\phi_{n}\in D(A)$ for all $n\in\ZZ$ and $A\phi_{n}=\lambda_{n}\phi_{n}$ for a complex sequence  $(\lambda_n)_{n}$ in a left-half plane of $\mathbb{C}$.  This implies that $A$ generates a strongly continuous semigroup $(T(t))_{t\ge0}$ with $T(t)\phi_{n}=\mathrm{e}^{\lambda_{n}t}\phi_{n}$ for all $n\in \ZZ$, $t\ge0$. 

In the above situation we say that {\emph{$A$ generates a diagonal semigroup with respect to the $q$-Riesz basis $(\phi_{n})_{n}$}. If the sequence $(\lambda_{n})_{n}$ lies in a vertical strip of the complex plane, we say that $A$ generates a {\emph{diagonal group}} with respect to the $q$-Riesz basis.

Note that the eigenvalues $(\lambda_n)_{n}$ lie in the open left-half plane $\mathbb{C}_{-}$ if and only if $(T(t))_{\ge0}$ is strongly stable, i.e.\ $\lim_{t\to\infty}\|T(t)x\|=0$ for all $x\in X$.
Without loss of generality   we may set $X=\ell^q$ and  
choose $\phi_{n}$ to be the $n$-th canonical basis vector of $\ell^q$.   
Further, we assume $U=\mathbb C$. It follows that every operator $B\in L(U,X_{-1})$ can be represented by a sequence $(b_n)_{n\in\ZZ}$ in
 \[X_{-1}\cong\left\{b\in \CC^{\mathbb{Z}}\colon \left(\frac{b_{n}}{\lambda_{n}-\lambda}\right)_{n\in\ZZ}\in \ell^{q} \right\}\]
  for some $\lambda$ in the resolvent set $\rho(A)$ of $A$. We use the analogous notation if the index set $\ZZ$ is replaced by $\NN$. 
  
We can
 link admissibility with the  boundedness of Laplace--Carleson embeddings: the following result was proved in \cite{jpp12} only for $Z=\mathrm{L}^{p}$ with $1\le p<\infty$, but the case $p=\infty$ and even $Z=\mathrm{L}^{\Phi}$ for some Young function $\Phi$ follows analogously.
\begin{proposition}[Theorem 2.1 in \cite{jpp12}]
\label{propequiv}
Let $q \ge 2$ and let $A:D(A)\subset X\rightarrow X$ generate a strongly stable diagonal semigroup $(T(t))_{t\ge0}$ with respect to a $q$-Riesz basis of $X$. 
Let $Z$ be an Orlicz space. The  operator $B\in L(U,X_{-1})$ is infinite-time $Z$-admissible for $(T(t))_{t \ge 0}$
if and only if 
the Laplace--Carleson embedding
\[\LL f(z):= \int_0^\infty e^{-zt} f(t) dt,\qquad z\in \CC_{+}\]
 induces a continuous mapping from $Z(0,\infty)$
into $\mathrm{L}^{q}(\CC_+,d\mu)$, where $\mu$ is the measure $\sum |b_k|^{q}\delta_{-\lambda_k}$.
\end{proposition}

Hence, in order to answer the above mentioned questions, we use the new embedding theorem for the Laplace--Carleson embedding, which were proved in the previous section and of independent interest.

The main results of this section are the following.

\begin{theorem}\label{mainsec3}
Let $q \ge 2$. If  $A:D(A)\subset X\rightarrow X$ generates a diagonal group with respect to a $q$-Riesz basis on $X$, 
then
 $$\mathfrak{B}_{\mathrm{L}^{\infty}}(A,\mathbb C^n)=\mathfrak{B}_{\mathrm{L}^{q/(q-1)}}(A,\mathbb C^n).$$
\end{theorem}
Clearly the case $p=2$ in Theorem \ref{mainsec3}  is already covered by Theorem \ref{thm1}. 
\begin{proof} We first mention, that it suffices to prove the results for $n=1$ (and apply e.g.\ \cite[Prop.~4]{JacoSchwZwar19} in the general). The statement then follows directly from Proposition \ref{propequiv} and Theorem \ref{thm:pqexp1}. 
\end{proof}
As explained in the introduction, we cannot expect that $\mathfrak{B}_{\mathrm{L}^{\infty}}(A,\mathbb C^n)$ equals $\mathfrak{B}_{\mathrm{L}^{p}}(A,\mathbb C^n)$ for some $p<\infty$ in general. 

The following result, however, shows that for diagonal semigroup generators $A$, at least every element $B$ in $\mathfrak{B}_{\mathrm{L}^{\infty}}(A,\mathbb C^n)$ is contained in $\mathfrak{B}_{\mathrm{L}^{\Phi}}(A,\mathbb C^n)$ for some N-function $\Phi$ depending on $B$.

We include also a resolvent characterization of $L^\infty$-admissibility in this case.

\begin{theorem}\label{thm:OrliczfromLinfty}
Let $q \ge 2$ and $A:D(A)\subset X\rightarrow X$ be the generator of a strongly stable diagonal semigroup $(T(t))_{t\ge 0}$ with respect to a $q$-Riesz basis and eigenvalues $(\lambda_{n})_{n\in\mathbb{Z}}$. 
Let $ \alpha > (q -4)/2$, and let $\mu$ denote the measure $\mu=\sum |b_k|^{q}\delta_{-\lambda_k}$.
Then the following are equivalent:
\begin{enumerate}
\item The operator $B\in L(\CC,X_{-1})$ is infinite-time $\mathrm{L}^{\infty}$-admissible 

\item 
\begin{equation}\label{eq:condLinfty}
		\sum_{n\in\ZZ}\sup_{\substack{I\subset i\mathbb{R}\\I\textnormal{ interval}}}\frac{\mu(Q_I\cap S_n)}{|I|^q}<\infty,
	\end{equation} 
where $Q_I$ is the Carleson square and $S_n$ the dyadic strip defined in Sec.~\ref{sec21}. 

\item 
\begin{equation}  \label{WeissLinf}
	  \int_0^\infty  t^{3 + 2 \alpha -q}  \sup_{\Re z =t} \| ( z\mathrm{I}-A)^{-(2 + \alpha)} B\|^2\, \mathrm{d}t< \infty.
	\end{equation}
\end{enumerate}
In this case there exists an N-function $\Phi$ such that $B$ is infinite-time $\mathrm{L}^{\Phi}$-admissible.

Moreover, $B$ is zero-class $\mathrm{L}^{\infty}$-admissible for $(T(t))_{t\ge 0}$.
\end{theorem}
\begin{proof}
 The equivalence of the first two statements follows from Proposition \ref{propequiv} and \ref{thm:CharacterizationLpToLq}, the equivalence with (\ref{WeissLinf}) from the identity
 $$
     \| (z\mathrm{I}-A)^{-(2 + \alpha)} B\|^2 = \int_{\CC_+} \frac{1}{|w + \bar z|^{4 + 2 \alpha}} d\mu(w)
 $$
 together with Theorem \ref{thm:CharacterizationLpToLq}.
 
 \color{black}
 
 The existence of a suitable N-function is guaranteed by Theorem \ref{thm:InfinityImpliesPhi}. 
 Finally, the zero-class $\mathrm{L}^{\infty}$-admissibility follows immediately from the $\mathrm{L}^{\Phi}$-admissibility by H\"older's inequality for Orlicz spaces. 
\end{proof}

\color{black}
Since (finite-time) admissibility remains invariant under for the shifted generator $A-cI$, $c\in\mathbb{R}$, we obtain the following consequence. 
\begin{corollary}\label{cor:diag} Let $q \ge 2$ and $A:D(A)\subset X\rightarrow X$ be the generator of a diagonal semigroup with respect to a $q$-Riesz basis. 
Then for every $B\in\mathfrak{B}_{\mathrm{L}^{\infty}}(A,\mathbb C^n)$ there exists an N-function $\Phi$ such that $B\in\mathfrak{B}_{\mathrm{L}^{\Phi}}(A,\mathbb C^n)$.
\end{corollary}

Finally we can formulate a characterization for $L^{\Phi}$-admissible operators for the specific N-function $\Phi_{\exp}(t)=\exp(t)-t-1$. 
 This complements existing characterizations of $\mathrm{L}^{p}$-admissible operators for diagonal semigroups, \cite{jpp12}.
\begin{theorem}\label{thm:admPhiexp}
Let $q \ge 2$ and $A:D(A)\subset X\rightarrow X$ be the generator of a strongly stable diagonal semigroup $(T(t))_{t\ge 0}$ with respect to a $2$-Riesz basis and  eigenvalues $(\lambda_{n})_{n\in\mathbb{Z}}$. 
Then 
\begin{multline*}
\mathfrak{B}_{\mathrm{L}^{\Phi_{\exp}}}(A,\CC) \\
=\{(b_{k})_{k\in\mathbb{N}}\in{L}(\CC,X_{-1})\mid\sum_{n=1}^\infty n^2\sup_{\substack{I\subset i\mathbb{R}\\I\textnormal{ interval}}}\frac{\mu(Q_I\cap S_n)}{|I|^2} +\sup_{\substack{I\subset i\mathbb{R}\\I\textnormal{ interval}\\|I|=2}}  \mu(Q_I)<\infty\},
\end{multline*}
where $\mu,Q_I, S_n$ are defined as in Theorem \ref{thm:OrliczfromLinfty}. 
\end{theorem}
\begin{proof}
Since $\mathfrak{B}_{\mathrm{L}^{\Phi_{\exp}}}(A,\CC)=\mathfrak{B}_{\mathrm{L}^{\Phi_{\exp}}}(A-cI,\CC)$ for any $c\in\RR$, we can assume that $A$ has only eigenvalues with real part less than $-2$. The result now follows by Proposition \ref{propequiv} and Theorem \ref{thm:FiniteT}.
\end{proof}
\begin{remark}
\begin{enumerate}
\item  Theorems \ref{thm:OrliczfromLinfty} and \ref{thm:admPhiexp} can be used to formulate analogous results for finite-dimensional input spaces, i.e. $B\in L(\CC^{n},X_{-1})$ for $n\in\NN$, by considering every ``component'' of $B$ separately, see also \cite[Prop.~4]{JacoSchwZwar19}.
 \item Theorem \ref{thm:OrliczfromLinfty} generalizes \cite[Thm.~4.1]{JNPS}  where the case of analytic diagonal semigroups was considered and thus condition \eqref{eq:condLinfty} is satisfied for all $B\in L(\CC^n,X_{-1})$. Also note that in those references, $q$ may more generally be chosen from $[1,\infty)$. On the other hand note that \cite[Thm.~4.1]{JNPS} was generalized to more general analytic semigroups which are not necessarily diagonal in \cite{JacoSchwZwar19}.
 \item Corollary \ref{cor:diag} also relates to the concept of {\it input-to-state stability}. More precisely, following the results in \cite{JNPS}, it shows that for linear systems described by diagonal semigroups with respect to a $q$-Riesz basis, the notions of input-to-state stability and integral input-to-state stability are equivalent. This answers partially an open question for linear infinite-dimensional systems, see e.g.\ \cite[Open Problem 3.22]{MironchenkoPrieur}.
\end{enumerate}
\end{remark}
	
\section*{Acknowledgements}
Part of this research was carried out within the project \emph{MOTADA}, jointly funded by the University of Hamburg and Lund University. ER, SP and FS are grateful for this support. ER 
was also supported by the Knut and Alice Wallenberg foundation, KAW 2016.0442. SP also gratefully acknowledges support by VR grant 2015-05552.

The authors thank the anonymous referee for his or her very careful reading of the paper, in particular for the suggestions concerning the literature on Orlicz spaces.

\end{document}